\documentclass[11pt]{article}

\usepackage{amsmath, amsfonts, amsthm, verbatim,
amssymb, dsfont, xcolor, enumerate}

\newcommand{\dd}{\mathrm{d}}
\newcommand{\R}{\mathbb{R}}
\newcommand{\N}{\mathbb{N}}

\newcommand{\E}{\mathbf{E}}
\newcommand{\p}{\mathbf{P}}
\newcommand{\ii}{\mathfrak{i}}
\newcommand{\Var}{\mathbf{Var}}
\newcommand{\bone}{\mathds 1}
\newcommand{\dint}{\int \hspace*{-5pt} \int}
\newcommand{\trint}{\int \hspace*{-5pt} \int \hspace*{-5pt} \int}
\DeclareMathOperator{\Leb}{Leb}

\theoremstyle{plain}
\newtheorem{lemma}{Lemma}
\newtheorem{theorem}{Theorem}

\newtheorem{corollary}{Corollary}

\theoremstyle{remark}

\newtheorem{remark}{Remark}

\usepackage{authblk}

\title{Almost sure growth of integrated supOU processes}
\author[1]{Danijel Grahovac\thanks{dgrahova@mathos.hr}}
\author[2]{P\'eter Kevei\thanks{kevei@math.u-szeged.hu}}
\affil[1]{School of Applied Mathematics and Informatics, J. J. Strossmayer University of Osijek, Trg Ljudevita Gaja 6, 31000 Osijek, Croatia}
\affil[2]{Bolyai Institute, University of Szeged, Aradi v\'ertan\'uk tere 1, 6720 Szeged, Hungary}

\date{}

\begin{document}

\maketitle

\begin{abstract}
Superpositions of Ornstein-Uhlenbeck processes allow a flexible dependence structure, including 
long range dependence for OU-type processes. Their complex asymptotics are governed by three effects:
the behavior of the L\'evy measure both at infinity and at zero, and the behavior at zero of the measure governing the dependence. We establish almost sure rates of growth depending on the characteristics of the process and prove a Marcinkiewicz--Zygmund type SLLN for the integrated process.

\textit{AMS 2010 Subject Classiﬁcations:} 60G17, 60F15, 60G55, 60G57

\end{abstract}

\medskip

\section{Introduction}

A superposition of Ornstein-Uhlenbeck type processes (\textit{supOU process}) 
\cite{Barn} is a strictly stationary process $(X(t))_{t\in \R}$ given by 
\begin{equation}\label{eq:supOU}
X(t)= \dint_{(0,\infty) \times (-\infty,t]} e^{-x( t - s)} \Lambda(\dd x, \dd s),
\end{equation}
where $\Lambda$ is a homogeneous infinitely divisible independently scattered random measure on $(0,\infty) \times \R$, 
meaning that for every sequence 
$A_1,A_2,\dots$ of disjoint bounded sets from $\mathcal{B} \left((0,\infty) \times \R\right)$, the random variables $\Lambda(A_1)$, $\Lambda(A_2)$, $\dots$ are independent. 
Moreover, for bounded
$A \in \mathcal{B} \left((0,\infty) \times \R\right)$ we have
\begin{equation}\label{eq:Lambdace}
\begin{split}
& \log \E e^{ \ii \theta \Lambda(A) }  \\  
& = \left( \pi \times \Leb \right) (A) \left(\ii \theta a 
-\frac{\theta^{2}}{2} b  + \int_{\R} \left( e^{\ii \theta z} - 1 -
\ii \theta z \bone(|z| \leq 1)\right) \lambda(\dd z) \right),
\end{split}
\end{equation}
where $\pi$ is a measure on $(0,\infty)$ such that $m_{-1}(\pi):=\int_{(0,\infty)} x^{-1} \pi(\dd x) < \infty$, and $(a,b,\lambda)$ is a L\'evy-Khintchine triplet of some infinitely divisible distribution such that $\int_{|z|>1} \log |z| \lambda (\dd z ) < \infty$. The \textit{generating quadruple} $(a,b,\lambda,\pi)$ completely determines the distribution of the supOU process.

The one-dimensional marginal distribution of the supOU process corresponds to that of Ornstein--Uhlenbeck type
process driven by the L\'evy process $L$ with characteristic triplet $(a,b,\lambda)$. In particular, 
\begin{equation*}
X(t) \overset{d}{=} m_{-1}(\pi) \int_{0}^{\infty} e^{-s} L(\dd s),
\end{equation*}
hence the marginal distribution is controlled by the triplet $(a,b,\lambda)$. L\'evy driven Ornstein--Uhlenbeck processes always have exponentially decaying correlations,
whenever their second moment exists. However, supOU processes have more complex dependence structure since the correlation 
function of $X$ is the Laplace transform of $x^{-1}\pi(\dd x)$, 
provided that $\E X(t)^2<\infty$. See 
\cite{Barn,barndorff2018ambit,barndorff2011multivariate} for more details.

In this paper we investigate the almost sure behavior of the \textit{integrated supOU process}
\begin{equation}\label{eq:integratedsupOU}
X^*(t) = \int_0^t X(u) \dd u.
\end{equation}
In our main results, Theorems \ref{thm:fv} and \ref{thm:infv} below, we prove a Marcinkiewicz--Zygmund type strong law of 
large numbers for the integrated process. The classical Marcinkiewicz--Zygmund theorem, see e.g.~\cite[Theorem 6.7.1]{Gut}, states that if iid random variables $Y_1,Y_2,\dots$ have finite moment of order $\gamma \in (1,2)$, then the centered partial sum 
$S_n=\sum_{i=1}^n (Y_i - \E Y_i)$ is $o(n^{1/\gamma})$ almost surely. For the same to hold for $X^*$, the main condition is that
\begin{equation} \label{eq:gamma-cond:intro}
\dint_{(0,\infty) \times \R} \frac{|z|^\gamma}{x^\gamma} \bone ( |z| > x) 
\pi(\dd x) \lambda(\dd z)   < \infty.
\end{equation}
This implies that $\E |X^*(1)|^\gamma < \infty$, however, \eqref{eq:gamma-cond:intro} is stronger than the moment assumption alone, as it is also affected by the behavior of $\pi$ and $\lambda$ near 0. On the other hand, in the purely Gaussian case, a law of iterated logarithm holds (Theorem \ref{thm:GaussianLIL}).

The motivation for studying the almost sure rate of growth comes from limiting properties established recently in a series of papers \cite{GLST2019,GLT2019Limit,GLT21}. The integrated supOU processes satisfy distributional limit theorems of the form
\begin{equation}\label{eq:intro:limittheorem}
\left\{ \frac{X^*(Tt) - \E X^*(Tt)}{T^H \widehat{\ell}(T)} \right \} 
\stackrel{\text{fdd}}{\longrightarrow} 
\left\{ Z(t) \right\}, \quad \text{ as } T \to \infty,
\end{equation}
with some slowly varying function $\widehat{\ell}$, $H \in (0,1)$ and limiting process $Z$ depending on the generating quadruple $(a,b,\lambda,\pi)$. 
Despite satisfying limit theorems, integrated supOU processes exhibit \textit{moment intermittency} (see e.g.~\cite{carmona1994parabolic,CK2019,gartner2007,Khos14,Khos17,zel1987}). Namely, their higher order moments may grow faster that those of the limiting process, i.e.
\begin{equation}\label{eq:intro:momsupOU}
\lim_{t\to \infty} \frac{\log \E |X^*(t) - \E X^*(t)|^q}{\log t} = \begin{cases}
H q, & 0\leq q\leq \frac{\alpha}{1-H},\\
q-\alpha, & q\geq \frac{\alpha}{1-H},
\end{cases}
\end{equation}
where $H$ is the self-similarity parameter of the limiting process $Z$ in \eqref{eq:intro:limittheorem} and $\alpha$ is parameter related to behavior of $\pi$ near zero. Intermittency implies that for any $\varepsilon>0$ and $\delta>0$ we eventually have \cite{GLT2022frac}
\begin{equation}\label{eq:intro:ldpstatement}
 \p \left( |X^*(t) - \E X^*(t)| > c t^{1-\varepsilon} \right) \geq t^{-\alpha-\delta}.
\end{equation}
Therefore, the probability that the normalized integrated supOU process
exhibits increasingly large values decreases as a power function of time. 
This small fraction of probability space where $X^*(t)$ is unduly large provides
the main contribution to higher order moments resulting in moment intermittency 
\eqref{eq:intro:momsupOU}. In particular, classical large deviation results with 
exponentially decaying probabilities do not typically hold in limit theorems for supOU processes \cite{GLT2022frac}.

These limiting results suggest that $X^*$ has a similar behavior of moments as the sequence of independent random variables $(Y(n))_{n\in \N}$ given by
\begin{equation*}
Y(n) = \begin{cases}
n^H,& \text{ with probability } 1-n^{-\alpha},\\
n,& \text{ with probability } n^{-\alpha}.\\
\end{cases}
\end{equation*}
Moreover, for $\alpha<1$ probabilities $\p (Y(n) \geq n)$ (as those in \eqref{eq:intro:ldpstatement}) are not summable. By the second 
Borel--Cantelli lemma $Y(n)\geq n$ happens infinitely often. 

The main motivation behind this paper is to investigate whether the same is true for paths of integrated supOU processes. Namely, we wonder whether the normalized integrated supOU process exhibits \textit{physical intermittency}, meaning that at large times it exhibits increasingly tall peaks (see e.g.~\cite{carmona1994parabolic,CK,Khos23,Khos20,zel1987}). 
We show that this not so. In particular, 
for any $\varepsilon>0$, we have that
\begin{equation*}
\lim_{t\to \infty} \frac{X^*(t)- \E X^{*}(t)}
{t^{H+\varepsilon}} = 0 \quad \text{a.s.}
\end{equation*}
We conclude that the paths of 
normalized integrated supOU process do not exhibit increasingly large values infinitely 
often, even though \eqref{eq:intro:momsupOU} and \eqref{eq:intro:ldpstatement} would 
suggest differently at first. This can be compared with the recently observed phenomenon 
of \textit{dissipation} in SPDEs \cite{Khos23,Khos20}.

Our results provide a rare example of a law of the iterated logarithm type and limsup type results for processes with strong dependence. Such results are typically known only for Gaussian \cite{orey1972,watanabe1970} or self-similar processes, like, for example, the law of iterated logarithm for fractional Brownian motion, limsup results for $\alpha$-stable L\'evy processes \cite{Sato},  linear fractional stable process \cite{takashima1989sample}, or for general self-similar processes (see \cite{kono1983iterated,takashima1989sample}). Beyond the class of self-similar processes, let us mention that for L\'evy processes the rate of growth depends on the regular variation index of the L\'evy measure at infinity (see \cite[Section 48]{Sato} and the references therein).

One may view supOU processes as aggregation of L\'evy driven Ornstein-Uhlenbeck processes with random mean reverting parameter. Such models have been extensively studied ever since the seminal paper \cite{granger1980} when it was realized that by aggregating simpler processes one may obtain complex dependence properties like long range dependence (see e.g.~\cite{leipus2003,oppenheim2004,philippe2014} and the references therein). However, law of iterated logarithm type results for such processes have been unknown so far. Moreover, distributional limit theorems for  supOU processes are similar to that of trawl processes \cite{barndorff2018ambit,GLT18trawl,pakkanen2021,talarczyk2020}. Hence, our results may pave the way for establishing similar pathwise results for these classes of processes. 

From a modeling perspective, supOU processes offer the appealing feature of allowing the marginal distribution and the dependence structure to be modeled independently. They have been used in various applications, mainly in finance where they provide a model for stochastic volatility (see e.g.~\cite{bn2013supOUSV,curato2019,fuchs2013}). In this context, the integrated process \eqref{eq:integratedsupOU} is the integrated volatility. SupOU processes have also been applied in other areas, like e.g.~in astrophysics \cite{kelly2013active}.

The paper is organized as follows. Section \ref{sec:main} contains the main results. In Section \ref{sec:pre} we give some preliminary facts needed for the proofs and also discuss moments and tail behavior of the integrated process. The proofs of the main results are given in Sections \ref{sec:proof-fv} and \ref{sec:proof-infv}. Section \ref{sec:proof-Gaussian}
contains the proof of the Gaussian case. In Appendix A we prove a stochastic Fubini 
theorem and in Appendix B we provide an exotic example of measures $\pi$ and $\lambda$.

Throughout the paper $c$ will denote a constant that may change from row to row. For any measure $Q$ on $(0,\infty)$ we denote 
\begin{equation*}
m_p ( Q ) = \int_{(0,\infty)} x^{p} Q (\dd x),
\end{equation*}
and for its tail we write
\begin{equation*}
\overline Q(r) = Q((r,\infty)), \quad r > 0.
\end{equation*}

\section{The almost sure growth}\label{sec:main}

The random measure $\Lambda$ in \eqref{eq:supOU} has a modification with the 
L\'evy--It\^{o} decomposition (see \cite[Theorem 4.5]{pedersen2003levy} and 
\cite[Theorem 2.2]{barndorff2011multivariate}). We will work with this modification 
so that 
\begin{equation}\label{eq:Lambda}
\begin{split}
\Lambda(\dd x, \dd s)  = & a \pi(\dd x) \dd s  +  \Lambda^G (\dd x, \dd s) 
+\int_{|z|\leq 1} z  (\mu - \nu) ( \dd x, \dd s, \dd z) \\
& + \int_{|z|>1} z  \mu ( \dd x, \dd s, \dd z),
\end{split}
\end{equation}
where $\Lambda^G$ is Gaussian random measure defined as in \eqref{eq:Lambdace} with generating quadruple $(0,b,0,\pi)$ and $\mu$ is a Poisson random measure on $(0,\infty) \times \R \times \R$ with intensity measure $\nu (\dd x, \dd s, \dd z) = \pi(\dd x) \dd s \, \lambda(\dd z)$.
Furthermore, the existence of a measurable modification of the supOU process 
$(X(t))_{t \in \R}$ in \eqref{eq:supOU} follows from \cite[p.~926]{bnbasse2011}. 
In the following, we will work with such a modification. However, we note that 
c\`adl\`ag modification does not necessarily exist, 
see \cite[Proposition 1]{GK}.

\begin{remark}
The supOU processes were introduced in \cite{Barn} with a slightly different 
parametrization ($-x t + s$ in the exponent in \eqref{eq:supOU}). Representation 
\eqref{eq:supOU} is a moving 
average representation, which was first used in \cite{Fasen05} 
(see also \cite{FasenKlupp}). 
The conditions $\int_{(0,\infty)} x^{-1} \pi(\dd x) < \infty$ and 
$\int_{|z|>1} \log |z| \lambda (\dd z ) < \infty$ are necessary and sufficient for the 
existence of $X$ by Theorem 3.1 in \cite{Barn}.
See also \cite[Proposition 2.1]{FasenKlupp} and \cite{barndorff2011multivariate} 
with the unnecessary assumption that $\pi$ is a probability measure.
\end{remark}

The integrated process \eqref{eq:integratedsupOU} may be written in the form
\begin{equation*}
X^*(t)  = \int_0^t X(u) \dd u = \dint_{(0,\infty) \times \R} f_t(x,s)  \Lambda(\dd x, \dd s),
\end{equation*}
where
\begin{equation*}
f_t(x, s) = 
\begin{cases}
x^{-1} (1 - e^{-x t}) e^{x s}, & s \leq 0, \\
x^{-1} (1 - e^{-x (t - s)}) \bone(s \leq t), & s > 0.
\end{cases}
\end{equation*}
The use of the stochastic Fubini theorem is justified in 
\cite[Theorem 3.1]{bnbasse2011} and
\cite[Lemma 4.1]{GLT21} (see the Appendix for more details). 
By \eqref{eq:Lambda}, we have
\[
\begin{split}
X^*(t) & = a \int_{(0,\infty)} \int_{\R} f_t(x,s)  \pi(\dd x) \dd s 
+  \int_{(0,\infty)} \int_{\R} f_t(x,s) \Lambda^G(\dd x, \dd s)\\
& \quad +  \int_{(0,\infty)} \int_{\R} \int_{|z|\leq 1} z f_t(x,s)  
(\mu - \nu) ( \dd x, \dd s, \dd z) \\
& \quad + \int_{(0,\infty)} \int_{\R} \int_{|z|>1} z f_t(x,s) 
\mu ( \dd x, \dd s, \dd z).
\end{split}
\]
Since $\int_{(0,\infty)} \int_{\R} f_t(x,s)  \pi(\dd x) \dd s  = t m_{-1}(\pi)$, the behavior of the deterministic part is clear. 
Hence, in the following we assume that $a=0$ whenever 
$\int_{|z| \leq 1} |z| \lambda(\dd z )  = \infty$, and 
when $\int_{|z|\leq 1} |z| \lambda(\dd z)<\infty$, then we assume
$a = \int_{|z|\leq 1} z \lambda(\dd z)$, so that there is no centering. 
In particular, if $\int_{|z|>1} |z| \lambda(\dd z)<\infty$, then
\begin{equation}\label{eq:meanofX}
\E X^*(t) =\begin{cases}
t m_{-1}(\pi) \int_{\R} z \lambda(\dd z), & 
\text{ if } \int_{|z|\leq 1} |z| \lambda(\dd z)<\infty,\\
t m_{-1}(\pi) \int_{|z|>1} z \lambda(\dd z), & 
\text{ if } \int_{|z|\leq 1} |z| \lambda(\dd z)=\infty.
\end{cases}
\end{equation}
Moreover, we assume for the moment that there is no Gaussian component in the generating quadruple.

The technique of the proof depends on whether the Poisson random measure needs 
compensation or not. Hence, we split the results in these two cases.

\begin{theorem}\label{thm:fv}
Assume that $a = \int_{|z|\leq 1} |z| \lambda(\dd z)<\infty$ and $b=0$.
\begin{enumerate}
\item If for some $\gamma \in (0,2)$ 
\begin{equation} \label{eq:gamma-cond}
\dint_{(0,\infty) \times \R} \frac{|z|^\gamma}{x^\gamma} \bone ( |z| > x) 
\pi(\dd x) \lambda(\dd z)   < \infty,
\end{equation}
and if $\gamma=1$ additionally $m_{-1-\delta} (\pi) < \infty$ for some $\delta>0$, then 
\[
\lim_{t \to \infty} \frac{X^*(t) - \bone(\gamma \geq 1) \E X^*(t)}{t^{1/\gamma}} = 0 \quad \text{a.s.}
\]
\item If \eqref{eq:gamma-cond} holds for $\gamma  = 2$, then
\[
\limsup_{t \to \infty} 
\frac{|X^*(t) - \E X^*(t)|}{\sqrt{2 t \log \log t}} = 
\sqrt{\Var X^*(1)} \quad \text{a.s.}
\]
\end{enumerate}
\end{theorem}

From Corollary \ref{cor:X*moments} below it follows that \eqref{eq:gamma-cond}
implies $\E |X^*(1)|^\gamma < \infty$, which is the moment assumption in the classical Marcinkiewicz--Zygmund theorem (see e.g.~\cite[Theorem 6.7.1]{Gut}).

For the $\int_{|z|\leq 1} |z| \lambda(\dd z)=\infty$ case, we introduce the following indices
\begin{equation} \label{eq:BG}
\begin{split}
& \beta_0 = 
\inf \left\{ \beta \geq  0: \, \int_{|z|\leq 1} |z|^\beta \lambda(\dd z) < \infty \right\}, \\
& \eta_\infty = 
\sup \left\{ \eta \geq  0: \, \int_{|z|>1} |z|^{\eta} \lambda(\dd z) < \infty \right\}, \\
& \alpha_0 = 
\sup \left\{ \alpha \geq 0: \, \int_{(0,1]} x^{-\alpha - 1} \pi (\dd x) < \infty \right\}.
\end{split}
\end{equation}
The constant $\beta_0$ is the Blumenthal--Getoor index of the L\'evy measure $\lambda$ and $\eta_\infty$ is the tail index of the L\'evy measure $\lambda$, which also corresponds to the tail index of the marginal distribution \cite[Proposition 3.1]{FasenKlupp}. 
Assumption $\int_{|z| \leq 1} |z| \lambda(\dd z) = \infty$ implies that 
$\beta_0 \geq 1$, while $\beta_0 \leq 2$, as $\lambda$ is a L\'evy measure.

In what follows, we put $\beta = \beta_0$ if $\int_{|z|\leq 1} z^{\beta_0} \lambda(\dd z) 
< \infty$, otherwise $\beta > \beta_0$ arbitrarily close. 
Similarly, $\eta \leq \eta_\infty$, $\alpha \leq \alpha_0$ with equality if 
the corresponding quantity is finite.

\begin{theorem}\label{thm:infv}
Assume that $\int_{|z|\leq 1} |z| \lambda(\dd z)=\infty$, $a=0$ and $b=0$.
\begin{enumerate}
\item If $\beta \leq 1 + \alpha$, assume that \eqref{eq:gamma-cond} holds for some $\gamma \in (0,2]$ and if $\gamma=1$ additionally $m_{-1-\delta} (\pi) < \infty$ for some $\delta>0$. Then 
\[
\limsup_{t \to \infty} \frac{
|X^*(t) - \bone(\gamma \geq 1) \E X^*(t)|}{t^{1/\gamma} \log t} \leq 1  \quad \text{a.s.}
\]
\item If $\beta \geq 1+\alpha$, 
assume that
\begin{equation} \label{eq:gamma-cond-2}
\dint_{(0,\infty) \times \R} \frac{|z|^\gamma}{x^\gamma} \bone ( |z| > x, |z| > 1) 
\pi(\dd x) \lambda(\dd z)   < \infty
\end{equation}
holds with $\gamma < 1/(1 - \alpha/\beta)$. Then
\[
\lim_{t \to \infty} \frac{X^*(t) - \bone(\gamma \geq 1) \E X^*(t)}{t^{1/\gamma}} = 0 \quad \text{a.s.}
\]
\end{enumerate}
\end{theorem}

To understand the assumptions, it is worth discussing how \eqref{eq:gamma-cond} is related to indices $\alpha$, $\beta$ and $\eta$. If \eqref{eq:gamma-cond} holds, then
\begin{equation*}
\int_{(0,1]} x^{-\gamma} \int_{|z|>1} |z|^\gamma \lambda(\dd z) \pi (\dd x) < \infty,
\end{equation*}
which implies that $\gamma \leq 1+\alpha$ and $\gamma \leq \eta$. Furthermore, \eqref{eq:gamma-cond} with any $\gamma > 0$ implies that 
\eqref{eq:gamma-cond} holds with $\gamma = 0$, that is 
\begin{equation} \label{eq:gamma-0}
\int_{(0,\infty)} \lambda(\{ z: |z| > x \}) \pi(\dd x) = 
\int_{\R} \pi((0,|z|)) \lambda(\dd z) < \infty.
\end{equation}
Therefore, if $\lambda(\{ z: |z| > x \}) = h(x) x^{-\beta_0}$ 
or $\pi((0,z)) = h(z) z^{1 + \alpha_0}$ around 0, where $h$ is a function 
such that $\lim_{x \downarrow 0} h(x) x^\varepsilon = 0$ and 
$\lim_{x \downarrow 0} h(x) x^{-\varepsilon} = \infty$ for any $\varepsilon > 0$,
then \eqref{eq:gamma-0} implies $\beta \leq 1 + \alpha$. 
In particular, this holds if $h$ is slowly varying at 0, 
i.e.~$\lim_{x \downarrow 0} \tfrac{h(xy)}{h(x)} = 1$ for each $y > 0$ (see e.g.~\cite{BGT}). 
On the other hand, it is possible 
to construct rather exotic measures $\pi$ and $\lambda$ such that \eqref{eq:gamma-cond} holds but $\beta > 1 + \alpha$ (see the example in the Appendix). Further connections between \eqref{eq:gamma-cond} and indices $\alpha$, $\beta$ and $\eta$ can be seen from the proof of Corollary \ref{cor:lthm} below.

The following corollary summarizes the main results in terms of the bounds for the \textit{rate of growth} of the (centered) integrated process. For simplicity we assume that $\pi$ is a finite measure. The generalization to infinite measures resulting in a few more cases, is straightforward.

\begin{corollary}\label{cor:lthm}
Assume that $\pi$ is a finite measure. For the rate of growth of the integrated process 
\begin{equation*}
R(t)=\frac{\log |X^*(t)- \bone(\E X^*(t) < \infty) \E X^*(t)|}{\log t},
\end{equation*}
it holds a.s.~that
\begin{equation*}
\limsup_{t\to \infty} R(t) \leq \begin{cases}
\frac{1}{2}, & \text{if } \alpha \geq 1  \text{ and } \eta \geq 2,\\
\frac{1}{\eta}, & \text{if } \alpha \geq 1, \eta < 2, 
\text{ or } \alpha <1, \eta \leq 1+\alpha, \beta \leq 1 + \alpha \\
\frac{1}{1+\alpha}, & \text{if } \alpha <1, \eta > 1+\alpha
\text{ and } \beta \leq 1+\alpha,\\
1 - \frac{\alpha}{\beta}, & \text{if } \alpha <1, \eta >1+\alpha, 
\text{ and } 
\beta > 1+\alpha.\\
\end{cases}
\end{equation*} 
\end{corollary}

\begin{proof}
Write 
\[
\begin{split}
& \dint_{(0,\infty)\times \R} \frac{|z|^\gamma}{x^\gamma} \bone ( |z| > x) \pi(\dd x) 
\lambda(\dd z) \\
& = 
\int_{(0,1]} x^{-\gamma} \pi (\dd x) \left( \int_{|z| \in (x,1]} |z|^\gamma \lambda(\dd z) 
+ \int_{|z| > 1} |z|^\gamma \lambda(\dd z) \right) \\
&  \quad + 
\int_{(1,\infty)} x^{-\gamma} \pi (\dd x) \int_{|z| > x} |z|^\gamma \lambda(\dd z) \\
& =: J_1  + J_2 + J_3.
\end{split}
\]
Then 
$J_1 < \infty$ if $\gamma \leq 1 + \alpha$ and $\gamma \geq \beta$,
$J_2 < \infty$ if 
$\gamma \leq \eta \wedge (1 + \alpha)$, and 
$J_3 < \infty$ for $\gamma \leq \eta$ as $\pi$ is assumed to be a finite measure. We now consider each case separately.

($\alpha \geq 1$ and $\eta \geq 2$): Then $\beta \leq 1 + \alpha$ necessarily 
holds, and \eqref{eq:gamma-cond} holds with any $\gamma < 2$. Thus the claim follows 
from Theorem \ref{thm:fv} or \ref{thm:infv} part 1.

($\alpha \geq 1$, $\eta < 2$), or 
($\alpha < 1$, $\eta \leq 1 + \alpha$, $\beta \leq 1 + \alpha$): 
Then $J_2 < \infty$ for any $\gamma \leq \eta$,
and $J_1 < \infty$ for $\gamma \leq 1 + \alpha$ since $1 + \alpha \geq \beta$.
Finally, $J_3 < \infty$ for any $\gamma \leq \eta$.
Thus \eqref{eq:gamma-cond} holds with $\gamma \leq \eta$, and the claim follows 
again as above.

($\alpha < 1$, $\eta \geq 1 + \alpha$, $\beta \leq 1 + \alpha$):
Then $J_1 < \infty$ since $\gamma \leq + \alpha$ and $\beta \leq 1 + \alpha$,
while $J_2 < \infty$ and $J_3 < \infty$ since $\gamma \leq 1 + \alpha \leq \eta$.
Thus \eqref{eq:gamma-cond} holds with $\gamma \leq 1 + \alpha$, and the claim follows.

($\alpha < 1$, $\eta \geq 1 + \alpha$, $\beta \geq 1 + \alpha$):
Now we use Theorem \ref{thm:infv} part 2, and we have to check \eqref{eq:gamma-cond-2}.
In this case $J_1 = 0$, $J_2 < \infty$ for $\gamma \leq (1 - \tfrac{\alpha}{\beta})^{-1} 
\leq 1 + \alpha \leq \eta$, while 
$J_3 < \infty$ for $\gamma \leq \eta$.
\end{proof}

To show that the bounds obtained are optimal, we review details of the limit theorems in \eqref{eq:intro:limittheorem}. Suppose that $\pi$ has a density $p$ which is regularly varying at zero, that is
\begin{equation*}
p (x) \sim \alpha_0 \ell(x^{-1}) x^{\alpha_0}, \quad \text{ as } x \to 0,
\end{equation*}
for some $\alpha_0>0$ and $\ell$ slowly varying at infinity. If the integrated process has finite variance, then we have \eqref{eq:intro:limittheorem} if one of the following holds:
\begin{enumerate}
\item $\alpha_0 >1$, in which case $H=1/2$ and $Z$ is Brownian motion,
\item $b=0$, $\alpha_0 \in (0,1)$ and $\int_{|z|\leq 1} |z|^{1+\alpha_0}\lambda(\dd z)<\infty$, in which case $H=1/(1+\alpha_0)$ and $Z$ is $(1+\alpha_0)$-stable L\'evy process,
\item $b=0$, $\alpha_0 \in (0,1)$ and $\lambda \left( \{ x\in \R : |x|>z\} \right) \sim c z^{-\beta_0}$ as $z\to 0$ with $1+\alpha_0<\beta_0<2$, in which case $H=1-\alpha_0/\beta_0$ and $Z$ is $\beta_0$-stable process with dependent increments,
\item $b>0$ and $\alpha_0 \in (0,1)$, in which case $H=1-\alpha_0/2$ and $Z$ is fractional Brownian motion.
\end{enumerate}
For the infinite variance case, the type of the limit depends additionally on the regular variation index of the marginal distribution \cite{GLT2020multifaceted}.

By comparing with the bounds from Theorem \ref{thm:fv} and \ref{thm:infv}, one can see that the bounds are optimal, up to a slowly varying terms. We can conclude that the almost sure growth is of the same order as in the weak limits even though from \eqref{eq:intro:ldpstatement} one would expect differently.

Finally, for completeness we consider the purely Gaussian case, that is 
when $a=0$, $b>0$, $\lambda \equiv 0$. In this case, the following law of 
iterated logarithm follows from general results on Gaussian processes combined with the asymptotic analysis of variance. The proof is given in Section \ref{sec:proof-Gaussian}.

\begin{theorem}\label{thm:GaussianLIL}
Assume that $a=0$, $b>0$ and $\mu\equiv 0$, 
\begin{enumerate}
\item If $m_{-2}(\pi)<\infty$, then
\[
\limsup_{t\to \infty} \frac{|X^*(t)|}{\sqrt{2 b m_{-2}(\pi) t \log \log t} } = 1 \quad \text{a.s.}
\]
\item If $\pi$ has a density $p$ such that $p(x)\sim \alpha \ell(x^{-1}) x^{\alpha}$ for some $\alpha \in (0,1)$ and $\ell$ slowly varying at infinity, then
\[
\limsup_{t\to \infty} \frac{|X^*(t)|}{\widetilde{\sigma} \ell(t)^{\frac{1}{2}} 
t^{1-\frac{\alpha}{2}} \sqrt{2 \log \log t} } = 1 \quad \text{a.s.}
\]
where $\widetilde{\sigma}^2 = b \frac{\Gamma(1+\alpha)}{(2-\alpha)(1-\alpha)}$.
\end{enumerate}
\end{theorem}

\section{Preliminaries}\label{sec:pre}

As the positive and negative jumps of the Poisson random measure $\mu$ from \eqref{eq:Lambda} are independent, 
we will show in the proofs of Theorems \ref{thm:fv} and \ref{thm:infv} that it is enough to consider 
supOU processes for which this random measure has only positive jumps, 
i.e.~$\lambda((-\infty,0))=0$. Since we consider the Gaussian case separately, we assume 
that there is no Gaussian component.

The following statement is Theorem 4.1 in \cite{Barn}.
We prove it again here because of the explicit form of
the L\'evy measure $\eta^*_t$.

\begin{lemma}[{\cite[Theorem 4.1]{Barn}}] \label{lemma:X*-chf}
The random variable $X^*(t)$
is infinitely divisible with characteristic function
\[
\E e^{\ii \theta X^*(t)} =
\exp \left\{ \ii A \theta +
\int_{(0,\infty)} (e^{\ii \theta y} - 1 - \ii \theta y \bone(y \leq 1))
\eta_t^*(\dd y) \right\},
\]
where $A \in \R$ is an explicit constant and 
\[
\eta_t^*(B) = \nu \left( \{ (x,s,z): f_t(x,s) z \in B  \} \right).
\]
Furthermore, 
\[
\overline \eta_t^*(r) = \eta_t^*((r,\infty)) =
\int_{(0,\infty)} x^{-1} \pi (\dd x) \int_{\frac{x r}{1- e^{-x t}}}^\infty 
 \overline \lambda(z) \frac{1}{z - x r} \dd z.
\]
\end{lemma}

\begin{proof}
The first part follows from \cite[Theorem 2.7]{rajputrosinski1989}.
Note that for $r > 0$
\[
\begin{split}
& \{ (x,s,z): f_t(x,s) z > r \} \\
& = 
\left\{ (x, s, z) : z > \frac{x r}{1-e^{-x t}}, \ 
x^{-1} \log \frac{x r}{z (1-e^{-x t})} < s \leq 0 \right\} \\
& \quad \cup 
\left\{ (x, s, z) : z > \frac{x r}{1-e^{-x t}}, \ 
0 < s < t + x^{-1} \log 
\left( 1 - \frac{x r}{z} \right) \right\},
\end{split}
\]
thus 
\[
\begin{split}
& \overline \eta_t^*(r) \\
& = \int_{(0,\infty)} \pi (\dd x) \int_{(\frac{x r}{1- e^{-x t}}, \infty)} 
\lambda(\dd z) 
\frac{1}{x} \left( \log \frac{z (1-e^{-x t})}{x r}
+ tx + \log \left( 1 - \frac{x r}{z} \right)
\right) \\
& = \int_{(0,\infty)} x^{-1} \pi (\dd x) \int_{(\frac{x r}{1- e^{-x t}}, \infty)} 
\lambda(\dd z) 
\log \left( (e^{x t} -1 ) \left( \frac{z}{x r} -1 \right) 
\right) \\
& = \int_{(0,\infty)} x^{-1} \pi (\dd x) \int_{\frac{x r}{1- e^{-x t}}}^\infty 
\overline \lambda(z) \frac{1}{z - x r} \dd z,
\end{split}
\]
where the last equality follows after integration by parts combined with the fact 
that 
\[
\lim_{u \to \infty} \overline \lambda(u) \log (e^{xt}-1)(\tfrac{u}{xr} - 1) = 0.
\]
The latter limit holds since $\lim_{u \to \infty} \overline \lambda(u) \log u = 0$ 
by the integrability assumption $\int_{(1,\infty)} \log z \lambda(\dd z) < \infty$.
\end{proof}

Note that if $m_0(\pi) < \infty$, then $m_{-\beta}(\pi)< \infty$ for all
$\beta \in [0, 1]$.

\begin{corollary} \label{cor:X*moments}
Let $t > 0$ and $\beta > 0$. If $\E X^*(t)^\beta < \infty$, then $\int_{(1,\infty)} z^\beta \lambda(\dd z) < \infty$. On the other hand, if $\int_{(1,\infty)} z^\beta \lambda(\dd z) < \infty$, then $\E X^*(t)^\beta < \infty$ provided one of the following holds
\begin{enumerate}
\item $\beta\geq 1$,
\item $\beta \in (0,1)$ and $m_{-\beta}(\pi) < \infty$,
\item $\beta \in (0,1)$, $m_{-\beta}(\pi) = \infty$ and 
\end{enumerate}
\[
\int_{(1,\infty)} x^{-\beta} \pi(\dd x) \int_{x}^\infty \overline \lambda(z)
z^{\beta - 1} \dd z < \infty.
\]
\end{corollary}

\begin{proof}
By \cite[Theorem 25.3]{Sato},
$\E X^*(t)^\beta < \infty$ if and only if $\int_1^\infty r^{\beta -1} \overline 
\eta_t^*(r) \dd r < \infty$. By Lemma \ref{lemma:X*-chf}, 
and changing variables $u = \tfrac{xr}{z}$
\begin{equation} \label{eq:X*-mom-aux1}
\begin{split}
& \int_{1}^\infty r^{\beta - 1} \overline \eta_t^*(r) \dd r \\
& = \int_{1}^\infty \dd r \int_{(0,\infty)} \pi (\dd x) 
\int_{0}^\infty \dd z 
\frac{r^{\beta- 1} \overline  \lambda(z)}{x (z - x r)} 
\bone \left( z > \frac{x r}{1 - e^{-x t}} \right) \\
& = \int_{(0,\infty)} x^{-1} \pi (\dd x) 
\int_{x/(1-e^{-x t})}^{\infty} \dd z 
\overline \lambda(z) z^{\beta -1} x^{-\beta} 
\int_{x/z}^{1- e^{-x t}} \frac{u^{\beta -1} }{1-u} \dd u.
\end{split}
\end{equation}
The integral with respect to $u$ is an increasing function in $z$, therefore 
in each case
$\E X^*(t)^\beta < \infty$ implies 
$\int_1^\infty z^{\beta- 1} \overline \lambda(z) \dd z < \infty$,
which further implies
$\int_{(1,\infty)} z^\beta \lambda(\dd z) < \infty$.

For the converse, note that
\[
\sup_{0 \leq x \leq 1} 
x^{-\beta} \int_{0}^{1- e^{-x t}} u^{\beta- 1} (1-u)^{-1} \dd u
< \infty,
\]
while for $x> 1$
\[
\int_{0}^{1- e^{-x t}} u^{\beta- 1} (1-u)^{-1} \dd u = t \int_{0}^x (1-e^{-ty})^{\beta-1} \dd y \leq c x.
\]
Indeed, for $\beta\geq 1$ this follows from $1-e^{-ty}\leq 1$. For $\beta<1$ we have
\[
t \int_{0}^x (1-e^{-ty})^{\beta-1} \dd y \leq t \int_{0}^1 (1-e^{-ty})^{\beta-1} \dd y + t \int_{1}^x (1-e^{-t})^{\beta-1} \dd y \leq cx.
\] 
Substituting back into \eqref{eq:X*-mom-aux1} 
and using also that $u \geq 1 - e^{-u}$, we obtain
\[
\begin{split}
\int_{1}^\infty \overline \eta_t^*(r) r^{\beta - 1} \dd r 
& \leq c \int_{(0,1]} x^{-1} \pi (\dd x) \int_{1/t}^\infty \overline \lambda(z)
z^{\beta -1} \dd z  \\
& \quad + 
t \int_{(1,\infty)} x^{-\beta} \pi(\dd x) 
\int_{x}^\infty \overline \lambda(z) z^{\beta -1} \dd z,
\end{split}
\]
and the finiteness of the integral follows.
\end{proof}

Thus, in most cases the behavior of $\pi$ does not have effect on the 
existence of the moments. However, it does have effect on the 
existence of exponential moments.

\begin{lemma} \label{lemma:X*-expmoment}
Let $\varepsilon_0 = \sup \{ s \geq 0 : \pi((0,s)) = 0 \}$, $K_0(t) = \frac{1 - e^{-\varepsilon_0 t}}{\varepsilon_0}$, with $K_0(t) = t$ for $\varepsilon_0 = 0$. If
$\E e^{s X^*(t)} < \infty$, for $s > 0$ 
then $\int_{(1,\infty)} e^{zy} \lambda(\dd z) < \infty$
for any $y < K_0(t) s$. On the other hand, if 
$\int_{(1, \infty)} e^{z K_0(t) s} \lambda (\dd z ) < \infty$, then 
$\E e^{s X^*(t)} < \infty$.
\end{lemma}

\begin{proof}
From \cite[Theorem 25.3]{Sato} we see that $\E e^{s X^*(t)} < \infty$ if and only if
$\int_{(1,\infty)} e^{sr } \eta_t^*(\dd r) < \infty$, which holds 
if and only if $\int_1^\infty s e^{sr} \overline \eta_t^*(r) \dd r < \infty$.
Using Lemma \ref{lemma:X*-chf} and changing variables $u = \tfrac{xr}{z}$ we have
\[
\begin{split}
& \int_{1}^\infty e^{sr } \overline \eta_t^*(r) \dd r  = 
\int_1^\infty \dd r \int_{(0,\infty)} \pi (\dd x) 
\int_{\frac{x r}{1-e^{-x t}}}^\infty e^{sr} x^{-1} \overline \lambda(z) 
(z - x r)^{-1} \dd z \\
& = 
\int_{(0,\infty)} x^{-2} \pi (\dd x) \int_{0}^\infty \dd z
\int_{x/z}^{1-e^{-x t}} e^{s z u / x} \overline \lambda(z) (1-u)^{-1} \dd u
\bone \left(\frac{x}{z} < 1 - e^{-x t} \right) \\
& = 
\int_{(0,\infty)} x^{-2} \pi(\dd x) \int_0^1 (1-u)^{-1} \dd u 
\int_{0}^\infty \overline \lambda(z) e^{z su / x} \dd z 
\bone \left( \frac{x}{z} < u < 1- e^{-x t} \right).
\end{split}
\]
From the last formula we see that if the integral on the left-hand side is finite,
then $\int_1^\infty \overline \lambda(z) e^{z y} \dd z < \infty$
for $y < s (1-e^{-x t}) / x$, where $x \in \textrm{supp}(\pi)$.
Since $(1- e^{-x t})/x$ is decreasing in $x$, the first result follows.

Assuming $\int_{(1, \infty)} e^{z K_0(t) s} \lambda (\dd z )  < \infty$, 
we have 
\[
\int_{1}^\infty e^{sr } \overline \eta_t^*(r) \dd r \leq  
c \int_{(0,\infty)} x^{-2} \pi(\dd x) \int_0^{1- e^{-x t}} (1-u)^{-1} \dd u 
= c t. \qedhere
\]
\end{proof}

Lemma 4.2 in \cite{GLT21} states that if $\overline \lambda$ is $(-\gamma)$-regularly varying at infinity, which we denote by $\overline \lambda \in \mathcal{RV}_{-\gamma}$,
then $\E X^*(t)^q = \infty$ for all $q > \gamma$, which suggests power-law tail. 
The following statement makes this observation precise, by obtaining 
the exact asymptotics of the tail. We note that the multivariate regular variation result for the integrated process has already been established in \cite[Corollary 5.2]{moser2011}.

\begin{lemma}
Assume that $\overline \lambda \in \mathcal{RV}_{-\gamma}$ for some 
$\gamma > 0$. Further assume that there is $\varepsilon > 0$ such 
that 
\begin{equation} \label{eq:tail-ass}
\int_{(0,1)} x^{-1-\varepsilon} \pi(\dd x)
+ \int_{(1,\infty)} x^{\varepsilon - \gamma} \pi(\dd x)
< \infty.
\end{equation}
Then as $r \to \infty$
\[
\overline \eta_t^*(r) \sim \overline \lambda (r)
\int_{(0,\infty)} 
x^{-\gamma-1}
\int_{0}^{1-e^{-x t}} y^{\gamma -1} (1-y)^{-1} \dd y \, \pi(\dd x)  .
\]
\end{lemma}

Note that the second integral in \eqref{eq:tail-ass} is automatically finite 
(i) for any $\gamma > 0$ and $\varepsilon < \gamma$ if $m_0(\pi) < \infty$,
or (ii) for any $\gamma > 1$ and $\varepsilon < \gamma - 1$. 

\begin{proof}
Assuming $\overline \lambda(z) = \ell(z) / z^\gamma$ for some $\ell$ slowly varying at infinity, we have
\[
\begin{split}
\int_{\frac{x r}{1-e^{-x t}}}^\infty 
\frac{\overline \lambda(z)}{z-x r}
\dd z 
= (x r)^{-\gamma} \int_0^{1-e^{-x t}} y^{\gamma -1} (1-y)^{-1} 
\ell(x r/ y) \dd y.
\end{split}
\]
Note that 
\[
\lim_{x \downarrow 0} x^{-\gamma} \int_{0}^{1-e^{-x t}} 
y^{\gamma -1} (1-y)^{-1} \dd y = t^{\gamma} \gamma^{-1}.
\]
Thus, by Lemma \ref{lemma:X*-chf}
\[
\begin{split}
\overline \eta_t^*(r) & =
\frac{\ell(r)}{r^\gamma} \int_{(0,\infty)} 
x^{-\gamma-1}
\int_{0}^{1-e^{-x t}} y^{\gamma -1} (1-y)^{-1} \frac{\ell(x r/y)}{\ell(r)}
\dd y \, \pi(\dd x) \\
& \sim \overline \lambda(r) 
\int_{(0,\infty)} 
x^{-\gamma-1}
\int_{0}^{1-e^{-x t}} y^{\gamma -1} (1-y)^{-1} \dd y \, \pi(\dd x),\quad 
\text{ as } r\to \infty,
\end{split}
\]
where the last asymptotic equality follows from 
Potter's bounds \cite[Theorem 1.5.6]{BGT}, the integrability assumption, and 
Lebesgue's dominated convergence theorem.
\end{proof}

\section{Proof of Theorem \ref{thm:fv}}\label{sec:proof-fv}

We will first consider the case when $\lambda((-\infty,0))=0$. Note that through this section we assume that $\int_{(0,1]} z \lambda(\dd z) < \infty$ and that $a = \int_{(0,1]} z \lambda(\dd z)$, hence there is no centering. The integrated process then has the form
\begin{equation*}
X^*(t) = \int_{(0,\infty)} \int_{\R} \int_{(0,\infty)} z f_t(x,s)  
\mu( \dd x, \dd s, \dd z),
\end{equation*}
and we can decompose it as
\begin{equation}\label{eq:fv-decomposition}
\begin{aligned}
X^*(t) =& \int_{(0,\infty)} \int_{(-\infty,0]} 
\int_{(0,\infty)} z x^{-1} (1 - e^{-x t}) e^{x s}  \mu( \dd x, \dd s, \dd z) \\
 & + \int_{(0,\infty)} \int_{(0,t]} \int_{(0,\infty)} z x^{-1} \mu( \dd x, \dd s, \dd z)\\
&- \int_{(0,\infty)} \int_{(0,t]} \int_{(0,\infty)} z x^{-1} e^{-x (t - s)} \mu( \dd x, \dd s, \dd z) \\
=: & \, X_-^*(t) + X_{+,1}^*(t) - X_{+,2}^*(t).
\end{aligned}
\end{equation}
We now consider each of the terms separately. 
The term $X_{-}^*$ corresponding to negative jump times is easy to 
handle, and it is treated in Lemma \ref{lemma:Xminus-bound}.
The main term $X_{+,1}^*$ is a subordinator, a nonnegative increasing 
L\'evy process. We apply the classical Marcinkiewicz--Zygmund theorem, and 
its process version, to determine the almost sure growth rate of $X_{+,1}^*$.
This part is done in Lemma \ref{lemma:I1}. The difficult part is to handle 
the other negligible term $X_{+,2}^*$ for $\gamma \geq 1$. We need a fine 
analysis of the Poisson random measure $\mu$: we decompose the space 
$(0,\infty) \times (0,t] \times (0,\infty)$ to countably many pieces
according to the jump sizes and obtain sharp bounds on the number 
of points in each piece. This is done in Lemma \ref{lemma:I2}, with the 
necessary bounds in Lemma \ref{lemma:nu-r1r2} and Corollary \ref{cor:nu-r}.
This part of the proof is similar to the techniques used in \cite{CK}.

Let $(\xi_k, \tau_k, \zeta_k)_{k \geq 0}$ denote the points of the 
Poisson random measure $\mu$. Then 
\[
X^*_{-}(t) = \sum_{ \tau_k \leq 0} \frac{\zeta_k}{\xi_k} 
(1 - e^{-\xi_k t}) e^{\xi_k \tau_k}.
\]

\begin{lemma} \label{lemma:Xminus-bound}
Let $\gamma \in (0,2]$.
For $\gamma > 1$ suppose that  
$\int_{(0,1)} x^{-\gamma} \pi (\dd x)  < \infty$ and
$\int_{(0,\infty)} z \lambda(\dd z) < \infty$. Then 
\begin{equation*}
\lim_{t \to \infty} \frac{X_-^*(t)}{t^{1/\gamma}} = 0 \quad \text{a.s.}
\end{equation*}
\end{lemma}

\begin{proof}
Let first $\gamma \in (0,1]$ and note that since $1 - e^{-x} \leq x$ we have
\[
t^{-1} \, X_-^*(t) \leq t^{-1} \sum_{\tau_k < 0}  
\frac{\zeta_k}{\xi_k} \xi_k t e^{\xi_k \tau_k}
 \leq 
\sum_{\tau_k < 0}  {\zeta_k} e^{\xi_k \tau_k}.
\]
If the latter sum is finite a.s., then Lebesgue's dominated convergence theorem  gives the result. The latter sum exists if and only if 
\[
\trint_{(0,\infty)^3} (z e^{-sx} \wedge 1) \nu(\dd x, \dd s, \dd z) < \infty.
\]
Simple calculation gives that 
\[
\begin{split}
& \trint_{(0,\infty)^3} (z e^{-sx} \wedge 1) \nu(\dd x, \dd s, \dd z) \\
& = \int_{(0,\infty)} x^{-1} \pi(\dd x) \left( 
\int_{(0,1]} z \lambda(\dd z) + \int_{(1,\infty)} ( 1 + \log z) \lambda(\dd z)
\right),
\end{split}
\]
thus the statement follows.
\medskip

For $\gamma \in (1,2)$, since $X_-^*(t)$ is a Poisson integral, we have
\begin{equation*}
\begin{split}
\E X_-^*(t)  & =  
\trint \frac{z}{x} (1 - e^{-tx} ) e^{xs} \nu (\dd x, \dd s, \dd z)  \\
& = m_1(\lambda) \, \int_{(0,\infty)} x^{-2} (1 - e^{-tx} ) \pi (\dd x).
\end{split}
\end{equation*}
By the assumption and Lebesgue's dominated convergence theorem
\[
\begin{split}
\int_{(0,\infty)} x^{-2} (1 - e^{-tx} ) \pi (\dd x) 
& = t^{2 - \gamma} 
\int_{(0,\infty)} x^{-\gamma} \frac{1 - e^{-tx}}{(tx)^{2-\gamma}}  \pi 
(\dd x) \\
& = t^{2 - \gamma} o(1), \quad t \to \infty,
\end{split}
\]
as the fraction in the integrand is bounded and tends to 0.
Summarizing, as $t \to \infty$
\[
\E X_-^*(t)  = t^{2-\gamma} o(1).
\]
Therefore, by Markov's inequality
\[
\p ( X_-^*(t) > \varepsilon t^{1/\gamma} ) 
\leq \frac{\E X_-^*(t)}{\varepsilon t^{1/\gamma}}
\leq t^{2 - \gamma - \gamma^{-1}} o(1),
\]
and by the Borel--Cantelli lemma, for $a > 0$ such that 
$\gamma +\gamma^{-1} - 2 > 1/a$,
with $t_n = n^a$
\[
\lim_{n \to \infty} t_n^{-1/\gamma} X_-^*(t_n) = 0.
\]
The convergence now follows by monotonicity, as for $t \in [t_n, t_{n+1}]$
\[
t^{-1/\gamma} X_-^*(t) \leq t_n^{-1/\gamma} X_-^*(t_{n+1}) =
t_{n+1}^{-1/\gamma} X_-^*(t_{n+1}) \, \left( \frac{t_{n+1}}{t_n} \right)^{1/\gamma}
\to 0.
\]

The same computation shows that for $\gamma = 2$
\[
\lim_{t \to \infty} \E X^*_{-}(t) = m_1(\lambda) m_{-2}(\pi).
\]
Therefore the same reasoning works with $a = 3$. \qedhere
\end{proof}

\begin{remark}
Under the stronger assumption 
\begin{equation} \label{eq:X-cond}
\dint_{(0,\infty)^2} \left( \bone(z \leq x) \frac{z}{x^2} + 
\bone( z > x) \frac{1}{x} \left( 1 + \log \frac{z}{x} \right) \right) 
\pi(\dd x) \lambda(\dd z) < \infty, 
\end{equation}
$\sup_{t \geq 0} X_-^*(t) < \infty$ a.s. Indeed,
$\sup_{t \geq 0} X_-^*(t) \leq 
\sum_{\tau_k \leq 0} \frac{\zeta_k}{\xi_k}  e^{\xi_k \tau_k},$
which is finite a.s.~if and only if \eqref{eq:X-cond} holds.
\end{remark}

Next we handle the main term.

\begin{lemma} \label{lemma:I1}
Assume that for some $\gamma \in (0,2)$
\begin{equation} \label{eq:i1-cond}
\dint_{(0,\infty)^2} \left( \frac{z}{x} \right)^{\gamma} \bone ( z > x) 
\pi(\dd x) \lambda(\dd z)  < \infty.
\end{equation}
Then 
\begin{equation} \label{eq:i1-as}
\lim_{t \to \infty} \frac{X_{+,1}^*(t) - \bone ( \gamma \geq 1) t m_1(\lambda) m_{-1}(\pi)}{t^{1/\gamma}} = 0 \quad \text{a.s.} 
\end{equation}
If \eqref{eq:i1-cond} holds with $\gamma = 2$, then the law of iterated logarithm holds, i.e.
\[
\limsup_{t \to \infty} \frac{| X_{+,1}^*(t) - t m_1(\lambda) m_{-1}(\pi)|} {\sqrt{2 t \log \log t}} = \sqrt{m_{2}(\lambda) m_{-2}(\pi)}\quad \text{a.s.} 
\]
\end{lemma}

\begin{proof}
Note that $X_{+,1}^*(t)$ is a subordinator with characteristic function 
\[
\E e^{\ii u X_{+,1}^*(t)} = \exp  \left\{  t \int (  e^{\ii u y} - 1 ) \eta_1(\dd y)
\right\},
\]
where the L\'evy measure $\eta_1$ is given by
\[
\eta_1((r, \infty)) = 
\overline \eta_1(r) = \pi \times \lambda \left( \{ (x, z) : z > r x \} \right)
= \int_{(0,\infty)} \overline \lambda(r x) \pi ( \dd x).
\]
Thus $\E X_{+,1}^*(1)^\gamma < \infty$ if and only if 
$\int_{(1,\infty)} r^{\gamma} 
\eta_1(\dd r) < \infty$ which holds if and only if \eqref{eq:i1-cond} holds. Therefore, 
the Marcinkiewicz--Zygmund strong law of large numbers  implies that \eqref{eq:i1-as} 
holds whenever \eqref{eq:i1-cond} holds with $0 < \gamma < 2$. Indeed, if $\gamma < 1$, 
then by monotonicity the result follows from the usual partial sum version (\cite[Theorem 
6.7.1]{Gut}). For $\gamma \in [1,2)$, Theorem 2.1 in \cite{Tiefeng}, the 
process version 
of the Marcinkiewicz--Zygmund SLLN, gives the result since
\begin{equation*}
\E X_{+,1}^*(1) = \int_0^\infty \overline{\eta}_1(r) \dd r = m_1 (\lambda) m_{-1}(\pi).
\end{equation*}

The statement for $\gamma = 2$ follows from the law of iterated logarithm 
for L\'evy processes (see Proposition 48.9 in \cite{Sato}) and the fact that 
\[
\begin{split}
\Var (X_{+,1}^*(1))
= \trint_{(0,\infty)\times (0,1] \times (0,\infty)} 
\frac{z^2}{x^2} \nu(\dd x, \dd s , \dd z) 
=m_{2}(\lambda) m_{-2}(\pi),
\end{split}
\]
see e.g.~\cite[Theorem 2.7]{Kyprianou}.
\end{proof}

We note that in case of infinite mean ($\gamma < 1$) more precise integral 
tests are known, see \cite[Theorem III.13]{Bertoin}.

\medskip

It remains to consider $X_{+,2}^*$. For $t > 0$, $r > 0$ and $r_2 > r_1 > 0$ 
introduce the notation
\begin{equation*}
\begin{split}
& D(r,t)  = \left\{ (x,s,z) : 
\frac{z}{x} e^{-x(t-s)} > r, \ s \in (0, t), x > 0, z> 0
\right\}, \\
& D(r_1, r_2,t)  = 
\left\{ (x,s,z) : 
\frac{z}{x} e^{-x(t-s)} \in (r_1, r_2], \ s \in (0, t), x > 0, z> 0
\right\}.
\end{split}
\end{equation*}

\begin{lemma} \label{lemma:nu-r1r2}
For $t > 1$ and $0 < r_1 < r_2 < \infty$
\[
\begin{split}
\nu( D(r_1, r_2, t) ) & \leq 
t \int_{(0,t^{-1}]} \overline \lambda(r_1 x) \pi (\dd x) 
+ \log \frac{r_2}{r_1} \int_{(t^{-1}, \infty)} x^{-1} 
\overline \lambda(r_1 x)
\pi(\dd x) .
\end{split}
\]
\end{lemma}

\begin{proof}
Since $\frac{z}{x} e^{-x u} \in (r_1, r_2]$ for $u > 0$ if
and only if (i) $\frac{z}{x} \in (r_1, r_2]$ and $x^{-1} \log \frac{z}{r_1 x} > u$,
or (ii) $\frac{z}{x} > r_2$ and 
$x^{-1} \log \frac{z}{r_2 x} \leq  u < x^{-1} \log \frac{z}{r_1 x}$, we have
\begin{equation*}
\begin{split}
\nu(D(r_1, r_2, t)) & = \dint \bone ( r_1 x < z \le r_2 x) 
\left( t \wedge x^{-1} \log \frac{z}{r_1 x} \right) \\
& \ + \bone( z > r_2 x)  \! \left( t \wedge x^{-1} \log \frac{z}{r_1 x} - 
t \wedge x^{-1} \log \frac{z}{r_2 x} \right) 
\pi(\dd x) \lambda(\dd z)  \\
& \leq 
\dint \bone ( r_1 x < z ) 
\left( t \wedge x^{-1} \log \frac{r_2}{r_1} \right) \pi(\dd x) \lambda(\dd z).
\end{split}
\end{equation*} 
Bounding the expression in the bracket by $t$ 
for  $ x \in (0,t^{-1}]$ and by $x^{-1} \log \tfrac{r_2}{r_1}$
for $x \in (t^{-1}, \infty)$, we obtain the statement.
\end{proof}

If $\gamma = 1$, we need an extra assumption:
\begin{equation} \label{eq:gamma1-ass}
\exists \delta > 0: \quad  m_{-1-\delta} (\pi) < \infty.
\end{equation}

\begin{corollary} \label{cor:nu-r}
Assume \eqref{eq:i1-cond} with $\gamma \in [1,2]$, and for $\gamma = 1$ also assume 
\eqref{eq:gamma1-ass}.
There exists $C = C(\pi, \lambda, \gamma)$ such that for any $0 \leq a < b$
\[
\nu( D(n^{a}, n^{b}, n) ) \leq 
\begin{cases}
C n^{2 - \gamma - a} \log n, & \gamma > 1, \\
C n^{1 - a - \delta/2}, & \gamma = 1 .
\end{cases}
\]
\end{corollary}

\begin{proof}
Note that \eqref{eq:i1-cond} with $\gamma \in [1,2]$ implies that $m_{-\gamma}(\pi)<\infty$ and $m_{1}(\lambda)<\infty$. For $\gamma > 1$ we have
\[
\begin{split}
n \int_{(0,n^{-1}]} \pi (\dd x) \int_{(n^a x, \infty)} \lambda(\dd z)  
& \leq n \int_{(0,n^{-1}]} n^{-a} x^{-1} \pi (\dd x) m_{1}(\lambda) \\
& \leq m_{1}(\lambda) 
n^{1- a} \int_{(0,n^{-1}]} x^{-\gamma} n^{1 -\gamma } \pi (\dd x)  \\
& =m_{-\gamma}(\pi) m_{1}(\lambda) n^{2 - \gamma - a} ,
\end{split}
\]
and similarly,
\[
\begin{split}
\int_{(n^{-1}, \infty)} x^{-1} \pi (\dd x) \int_{(n^a x, \infty)} \lambda(\dd z) 
& \leq n^{-a} \int_{(n^{-1}, \infty)} x^{-2} \pi(\dd x) m_{1}(\lambda) \\
& \leq m_{-\gamma}(\pi) m_{1}(\lambda) n^{2 - \gamma - a}.
\end{split}
\]
The statement now follows from Lemma \ref{lemma:nu-r1r2}. For $\gamma = 1$ the proof is the same, only in the $\dd \pi$ integral we need $x^{-1-\delta}$.
\end{proof}

\begin{lemma} \label{lemma:I2}
Suppose that $\gamma \in (0,1)$, or $\gamma \in [1,2]$ and \eqref{eq:i1-cond} holds. If 
$\gamma = 1$, assume additionally that \eqref{eq:gamma1-ass} holds. Then
\[
\lim_{t \to \infty} t^{- 1/\gamma} X_{+,2}^*(t) = 0 \quad \text{a.s.}
\]
\end{lemma}

\begin{proof}
If $\gamma < 1$ the statement follows from Lemma \ref{lemma:I1} since 
$X_{+,1}^*(t)\geq X_{+,2}^*(t)$. Therefore we assume that $\gamma \geq 1$.

Put $Y_n = \sum_{\tau_k \in [n-1, n)} \frac{\zeta_k}{\xi_k}$. Then 
$Y_1, Y_2, \ldots$ is an iid sequence with the same 
distribution as $X_{+,1}^*(1)$, thus
$\E Y_1^\gamma < \infty$ by \eqref{eq:i1-cond}. Therefore, 
\begin{equation} \label{eq:Ymax-conv}
n^{-1/\gamma} \max_{i \leq n} Y_i \to 0 \quad \text{a.s.}
\end{equation}
(see e.g.~\cite[Theorem 3.5.1]{EKM}). Furthermore, $\max_{\tau_k <n } \frac{\zeta_k}{\xi_k} \leq \max_{i \leq n} Y_i$.

First we prove that 
\begin{equation} \label{eq:I2-conv-subseq}
\lim_{n \to \infty} n^{-1/\gamma} X_{+,2}^*(n) = 0 \quad \text{a.s.}
\end{equation}
Assume that $\gamma > 1$.
Let $c_0 = \gamma + \gamma^{-1} -2 > 0$, and choose $a_0$ so that 
$2 - \gamma  < a_0 < 2 - \gamma + c_0/2$, 
and let $m = \lfloor 2 a_0 / c_0 \rfloor$, with $\lfloor \cdot \rfloor$ 
standing for the lower integer part.
Let $N$ be an integer large enough to have $N (2 - \gamma - a_0) < - 1$.
Define the sets 
\begin{equation} \label{eq:A-sets}
\begin{split}
& A_{n,0} = \{ \mu( D(n^{a_0}, n) ) > N \}, \\
& A_{n, i+1} = \{ \mu ( D(n^{a_{i+1}}, n^{a_i}, n)) > 2 C \log n \, 
n^{2 - \gamma - a_{i+1}} \}, \ i = 0, \ldots, m,
\end{split}
\end{equation}
where $a_{i+1} = a_i - c_0 / 2$ for $i=0,1,\ldots, m - 1$, $a_{m+1} = 0$, and 
$C = C(\pi, \lambda, \gamma)$ is the constant in Corollary \ref{cor:nu-r}.
With this choice $a_m = a_0 - m c_0/2 \in [0, c_0/2)$.

We use the Poisson tail bounds (see \cite[Lemma 3.1]{CK})
\begin{equation} \label{eq:Poisson-tail}
\begin{split}
& \p ( N_\lambda \geq n ) \leq \frac{\lambda^n}{n!}, \\
& \p ( N_\lambda \geq x ) \leq e^{-0.19 \, x}, \quad x \geq 2 \lambda,
\end{split}
\end{equation}
where $N_\lambda$ stands for a Poisson random variable with parameter $\lambda$.
By Corollary \ref{cor:nu-r} (with $a = a_0$ and $b= \infty$) and 
\eqref{eq:Poisson-tail}
\[
\p ( A_{n,0} ) \leq C n^{N(2 - \gamma - a_0)} (\log n)^N,
\]
and, similarly for $i = 1,2,\ldots, m+1$ using the second bound in 
\eqref{eq:Poisson-tail}
\[
\p ( A_{n,i} ) \leq \exp \left\{-0.19 \cdot 2 C \log n \, 
n^{2 - \gamma - a_{i}} \right\}.
\]
Clearly, all these bounds are summable, 
therefore, by the first Borel--Cantelli lemma 
the events $(A_{n,i})_{n \geq 1, i = 0,1,\ldots, m+1}$ 
occur finitely many times almost surely. 

The contribution of the
points $(\xi_k, \tau_k, \zeta_k)$ for which $\zeta_k / \xi_k e^{-\xi_k (n - \tau_k)}
\in (n^{a_{i+1}}, n^{a_i}]$, $\tau_k < n$, if $n$ is large enough so 
that $A_{n,i+1}$ does not occur, is for $i= 0,1,\ldots, m$
\begin{equation} \label{eq:A-contr}
\begin{split}
\sum_{\substack{0<\tau_k < n,\\ \zeta_k / \xi_k e^{-\xi_k (n - \tau_k)} \in (n^{a_{i+1}}, n^{a_i}]}} \frac{\zeta_k}{\xi_k} e^{-\xi_k (n - \tau_k)} 
& \leq n^{a_i} 2 C \log n \, n^{2 - \gamma - a_{i+1}} \\
& \leq 
2 C  \log n \,  n^{2 - \gamma + c_0/2},
\end{split}
\end{equation}
by the choice of $c_0$, where the last inequality is in fact equality for $i \neq m$
(as $a_m - a_{m+1} = a_m \in [0,c_0/2)$).

The events in \eqref{eq:A-sets} take care of the contributions larger than 1.
Next we deal with the small contributions. Introduce the events
\begin{equation} \label{eq:B-sets}
B_{n,k} = \{ \mu( D(2^{-k-1}, 2^{-k}, n) ) > b_{n,k} \}, \quad k= 0,1,\ldots
\end{equation}
with 
\begin{equation*}
\begin{split}
b_{n,k} & = 6 \log (n (k+1)) +  2 n \int_{(0,n^{-1}]} \pi(\dd x) 
\int_{(2^{-k-1} x , \infty)} \lambda(\dd z) \\
& \quad + 2 \int_{(n^{-1}, \infty)} x^{-1} \pi(\dd x) \int_{(2^{-k-1} x, \infty)} 
\lambda(\dd z).
\end{split}
\end{equation*}
By Lemma \ref{lemma:nu-r1r2} and \eqref{eq:Poisson-tail} we have
\[
\p ( B_{n,k} ) \leq \exp \left\{ -6 \cdot 0.19 \cdot  \log (n (k+1)) \right\}
\leq (n (k+1))^{-1.1},
\]
which is summable for $k \geq 0$, $n \geq 1$.
The first Borel--Cantelli lemma implies that 
all these events occur finitely many times a.s. Note that in 
$b_{n,k}$ the term $6 \log (n(k+1))$ ensures the summability of the probabilities even 
if the other terms in $b_{n,k}$ are small.

The sum of small contributions such that 
$\zeta_j / \xi_j e^{-\xi_j (n - \tau_j)} \in (2^{-k-1}, 2^{-k}]$, $\tau_j < n$,
for $k = 0,1, \ldots$ and $n$ large enough so that $B_{n,k}$ does not occur, is bounded by
\[
\begin{split} 
\sum_{k = 0}^\infty 2^{-k} b_{n,k} & \leq 
c \log n + 2 n \int_{(0,n^{-1}]}
\pi (\dd x) \sum_{k=0}^\infty 2^{-k} \int_{(2^{-k-1} x, \infty)} \lambda (\dd z) \\
& \quad +
2 \int_{(n^{-1}, \infty)} x^{-1}
\pi (\dd x) \sum_{k=0}^\infty 2^{-k} \int_{(2^{-k-1} x, \infty)} \lambda (\dd z).
\end{split}
\]
The infinite sum can be bounded as
\[
\begin{split}
\sum_{k=0}^\infty 2^{-k} \int_{(2^{-k-1} x, \infty)} \lambda (\dd z) 
& = \int_{(0,\infty)} \sum_{k=0}^\infty 2^{-k} \bone( z > 2^{-k-1} x ) \lambda(\dd z) \\
& \leq  2 \int_{(0,\infty)} \left( 1 \wedge z x^{-1} \right) \lambda(\dd z) 
\leq  c x^{-1} .
\end{split}
\]
Substituting back, and using  
$\int_{(0,n^{-1}]} x^{-1} \pi (\dd x) \leq n^{1-\gamma}
m_{-\gamma}(\pi)$, and 
\[
\int_{(n^{-1}, \infty)} x^{-2} \pi(\dd x) \leq n^{2 - \gamma} m_{-\gamma}(\pi),
\]
we have
\begin{equation} \label{eq:B-contr}
\sum_{k=0}^\infty 2^{-k} b_{n,k} \leq c \log n + c n^{2- \gamma}.
\end{equation}

Summarizing, by \eqref{eq:A-contr} and \eqref{eq:B-contr} 
the contribution of the terms in $X_{+,2}^*(n)$ can be bounded by
\[
X_{+,2}^*(n) \leq W + N \max_{1 \leq k \leq n} Y_k 
+ c \log n \,  m n^{2 - \gamma + c_0/2}
+ c  n^{2 -\gamma},
\]
where the random variable $W$ is coming from the Borel--Cantelli lemma, and 
the second term comes from $A_{n,0}$.
By \eqref{eq:Ymax-conv} 
all the terms are $o(n^{1/\gamma})$, thus \eqref{eq:I2-conv-subseq} follows.
\smallskip

For $t \in (n, n+1)$, using \eqref{eq:Ymax-conv}
\[
t^{-1/\gamma} X_{+,2}^*(t) \leq n^{-1/\gamma} \left( X_{+,2}^*(n) + Y_{n+1} \right) 
\to 0 \quad \text{a.s.}
\]
as claimed.

\medskip 
For $\gamma = 1$ the same idea works. Recall $\delta  > 0$ from \eqref{eq:gamma1-ass}, 
and let $c_0 = \delta / 2$, $a_0 \in ( 1 - \delta / 2, 1 - \delta /4)$, 
$m = \lfloor 4 a_0 /  \delta \rfloor$.
Choose an integer $N$ such that  $N (1 - a_0 - \delta / 2) < -1$.
Define the events $A_{n,i}$'s similarly as in \eqref{eq:A-sets}, 
\[
\begin{split}
& A_{n,0} = \{ \mu( D(n^{a_0}, n) ) > N \}, \\
& A_{n, i+1} = \{ \mu ( D(n^{a_{i+1}}, n^{a_i}, n)) > 2 C \, 
n^{1 - \delta/2 - a_{i+1}} \}, \ i = 0, \ldots, m,
\end{split}
\]
where $a_{i+1} = a_i - c_0 / 2$ for $i=0,1,\ldots, m - 1$, $a_{m+1} = 0$, 
and $C = C(\pi, \lambda, 1)$ is the constant in Corollary \ref{cor:nu-r}.
As above,
$A_n$'s occur finitely many times a.s., and the contribution is $o(n)$.
For the small jumps define $B_{n,k}$'s as in \eqref{eq:B-sets}. The only
difference is that by assumption \eqref{eq:gamma1-ass}
\[
\int_{(0,n^{-1}]} x^{-1} \pi (\dd x) \leq  n^{-\delta} m_{-1-\delta}(\pi),
\]
and
\[
\int_{(n^{-1}, \infty)} x^{-2} \pi (\dd x) \leq  n^{1-\delta} m_{-1-\delta}(\pi).\qedhere
\]
\end{proof}
\bigskip

\begin{proof}[Proof of Theorem \ref{thm:fv}] \ 
\begin{enumerate}
\item If $\lambda((-\infty,0))=0$, then the proof follows directly from Lemmas \ref{lemma:Xminus-bound}, 
\ref{lemma:I1}, \ref{lemma:I2} and decomposition \eqref{eq:fv-decomposition}. Note that the centering term in \eqref{eq:i1-as} is indeed $\E X^*(t)$ by \eqref{eq:meanofX}. If $\lambda((-\infty,0))\neq 0$, then we decompose $X^*(t)$ as in \eqref{eq:fv-decomposition} but separately for positive and negative jumps, that is considering $\mu\bone(z>0)$ and $\mu\bone(z<0)$ separately. This gives
\begin{equation*}
X^*(t) = X_-^{*,+}(t) - X_-^{*,-}(t) + X_{+,1}^{*,+}(t) - X_{+,1}^{*,-}(t) - X_{+,2}^{*,+}(t) + X_{+,2}^{*,-}(t).
\end{equation*}
For $\gamma \in (0,2)$
the statement now follows by applying Lemmas \ref{lemma:Xminus-bound}, \ref{lemma:I1} and \ref{lemma:I2} to positive and negative jump parts separately. 

\item If $\gamma  = 2$, then the proof follows as for $\gamma \in (0,2)$. In general, we decompose $X_-^{*}(t)$ and $X_{+,2}^{*}(t)$ as in the previous case. When divided by $\sqrt{2 t \log \log t}$ all of these terms converge to 0 a.s. The statement now follows from Proposition 48.9 in \cite{Sato} applied to $X_{+,1}^{*}(t)$.\qedhere
\end{enumerate}
\end{proof}

\section{Proof of Theorem \ref{thm:infv}}\label{sec:proof-infv}

We shall first prove the theorem by assuming that $\lambda$ is supported on $(0,1]$. At the very end we will then combine this with Theorem \ref{thm:fv} to prove  Theorem \ref{thm:infv}. 

If $\int_{(0,1]} z \lambda(\dd z) = \infty$, then centering is needed and we have
\[
X^*(t) = \trint_{(0, \infty)  \times (-\infty, t]  \times (0,1]}
f_t ( x,s) z (\mu - \nu) (\dd x, \dd s , \dd z) =: X_{+}^* (t) - X_{-}^* (t),
\]
where
\begin{align*}
X_{+}^* (t) &= \trint_{(0,\infty) \times (0,t] \times (0,1]}
\frac{z}{x} \left( 1 - e^{-x(t-s)} \right) (\mu - \nu) (\dd x, \dd s, \dd z),\\
X_{-}^* (t) &= \trint_{(0,\infty) \times (-\infty,0] \times (0,1]}
\frac{z}{x} \left( 1 - e^{-xt} \right) e^{xs} (\mu - \nu) (\dd x, \dd s, \dd z).
\end{align*}

The following two statements, Lemmas \ref{lemma:X+diffbound} and \ref{lemma:aux-conv},
ensure that the paths are sufficiently smooth,
so that it is enough to prove almost sure limits on the integers. Then we decompose 
the process further to a part with only small jumps $X_{\pm, <, \kappa}^*$
and with only large jumps $X_{\pm, >, \kappa}^*$. The small jump part 
is analysed via its moment generating function in Corollary \ref{cor:gamma>1},
where the necessary bounds are established in Lemma \ref{lemma:X<momgenfunc}.
This part ruins the optimality, as an extra logarithmic factor appears.
The large jump part is handled in Lemmas \ref{lemma:X>-0} and \ref{lemma:X>betanagy}.
The truncation idea resembles the proof of the classical Kolmogorov and
Marcinkiewicz--Zygmund strong law of large numbers, see e.g.~\cite{Gut}.

\begin{lemma} \label{lemma:X+diffbound}
There exists a modification of $X^*_{\pm}$ such that for any 
$\theta \in [0,1/2)$
\[
\sup_n \E \left( 
\sup_{s \neq t, s,t \in [n,n+1]} 
\frac{|X^*_{+}(t) - X^*_{+}(s)|
+|X^*_{-}(t) - X^*_{-}(s)|}{|t-s|^\theta} \right)^2 < 
\infty.
\]
\end{lemma}

\begin{proof}
For $t_1 < t_2$ we have
\[
\begin{split}
& X^*_{+}(t_2) - X^*_{+}(t_1) \\
& = 
\trint_{(0,\infty) \times (0,t_1] \times (0,1]} \frac{z}{x} 
\left( e^{-x(t_1 - s)} - e^{-x(t_2 -s)} \right) (\mu - \nu) (\dd x, \dd s, \dd z) \\
& \quad + 
\trint_{(0,\infty) \times (t_1, t_2] \times (0,1] } \frac{z}{x} 
\left( 1 -  e^{-x(t_2 -s)} \right) (\mu - \nu) (\dd x, \dd s, \dd z) ,
\end{split}
\]
therefore, by Theorem 1 in \cite{MR} 
\[
\begin{split} 
& \E \left( X^*_{+}(t_2) - X^*_{+}(t_1) \right)^2 \\
& \leq  2
\trint_{(0,\infty) \times (0,t_1] \times (0,1]} \frac{z^2}{x^2} 
e^{-2x(t_1 - s)} \left( 1 - e^{-x (t_2 - t_1) } \right)^2 \nu(\dd x, \dd s, \dd z) \\
& \quad 
+ 2 \trint_{(0,\infty) \times (t_1, t_2] \times (0,1]} \frac{z^2}{x^2} 
\left( 1 -  e^{-x(t_2 -s)} \right)^2  \nu (\dd x, \dd s, \dd z) =: I_{1} + I_{2}. 
\end{split}
\]
Writing $\Delta t = t_2 - t_1$, evaluating the $\dd s$ integral, and using that $1 - e^{-y} \leq y$, 
the first term on the right-hand side can be bounded as
\[
\begin{split}
I_{1} & \leq  
2 (\Delta t)^2  \dint_{(0,\infty) \times (0,1]} \frac{z^{2}}{2x} (1 - e^{-2xt_1})  
\lambda(\dd z) \pi(\dd x) 
\leq m_2(\lambda) m_{-1}(\pi)  (\Delta t)^2 .
\end{split}
\]
For the second term we have 
\[
\begin{split}
I_{2} & = 2 m_2(\lambda) \int_{(0,\infty)} x^{-2} \pi(\dd x) 
\int_0^{\Delta t} ( 1 - e^{-xu} )^2 \dd u \\
& \leq 2 m_2(\lambda)
\Delta t \int_{(0,\infty)} x^{-2} ( x \Delta t  \wedge 1)^2 \pi(\dd x) \\
& \leq 2 m_2(\lambda) m_{-1}(\pi) (\Delta t)^2.
\end{split}
\]
Summarizing
\begin{equation} \label{eq:X+diffbound}
\E ( X^*_{+}(t_2) - X^*_{+}(t_1))^2 \leq 
3 m_2(\lambda) m_{-1}(\pi) (\Delta t)^2.
\end{equation}
Similarly, for $X_{-}^*$ we have
\begin{equation} \label{eq:X-diffbound}
\begin{split}
& \E \left( X^*_{-}(t_2) - X^*_{-}(t_1) \right)^2 \\
& \leq  2
\trint_{(0,\infty) \times (-\infty,0] \times (0,1]} \frac{z^2}{x^2} e^{2 x s} 
e^{-2xt_1} \left( 1 - e^{-x \Delta t } \right)^2 \nu(\dd x, \dd s, \dd z) \\
& = m_2(\lambda) \int_{(0,\infty)} \frac{1}{x^3} e^{-2 x t_1} 
\left( 1 - e^{-x \Delta t } \right)^2 \pi(\dd x ) \\
& \leq m_2(\lambda) m_{-1}(\pi) (t_2-t_1)^2.
\end{split}
\end{equation}

The claim follows from a version of Kolmogorov's continuity theorem 
(\cite[Theorem 4.3]{Khos09}
with $[a,b] = [n,n+1]$, $\gamma = p = 2$, $\theta < 1/2$).
Note that \eqref{eq:X+diffbound} and \eqref{eq:X-diffbound} hold for any $t_1, t_2 \in (0,\infty)$, which 
implies that the bound above is uniform in $n$.
\end{proof}

\begin{lemma} \label{lemma:aux-conv}
Assume that for a process $(Y(t))_{t \in \R}$
\[
\sup_{n\in \N} \E \left( \sup_{t \in [n,n+1]} | Y(t)  - Y(n) | \right)^2 < \infty,
\]
and for an increasing function $a(t)$ for which 
\[
\sum_n a(n)^{-2} < \infty, \quad \text{and } \ 
\lim_{n \to \infty} \tfrac{a(n)}{a(n+1)} = 1,
\]
we have $\limsup_{n\to \infty} |Y(n)|  / a(n) \leq c$ a.s., with some $c \geq 0$. Then 
\[
\lim_{t \to \infty} \frac{|Y(t)|}{a(t)} \leq c \quad \text{a.s.}
\]
\end{lemma}

\begin{proof}
By Markov's inequality
\[
\p \left( \sup_{t \in [n,n+1] } | Y(t) - Y(n) | > a(n) \varepsilon \right)
\leq \frac{\E \left( \sup_{t \in [n,n+1]} | Y(t)  - Y(n) | \right)^2}
{a(n)^2 \varepsilon^2}, 
\]
thus the first Borel-Cantelli lemma implies the statement.
\end{proof}

We shall frequently use the following estimates
\begin{equation} \label{eq:int-pi-lambda-bounds}
\begin{split}
& \int_{(y,1]} z^{u} \lambda(\dd z) \leq 
y^{u-\beta} \int_{(0,1]} z^\beta \lambda(\dd z), \quad u \leq \beta, \\
& \int_{(0, y]} z^{u} \lambda(\dd z) \leq 
y^{u-\beta} \int_{(0,1]} z^\beta \lambda(\dd z), \quad u \geq  \beta, \\
& \int_{(y,1]} x^{-v} \pi(\dd x) \leq 
y^{1+\alpha - v} \int_{(0,1]} x^{-1-\alpha} \pi(\dd x), \quad v \geq 1 + \alpha, \\
& \int_{(0, y]} x^{-v} \pi(\dd x) \leq 
y^{1+\alpha - v} \int_{(0,1]} x^{-1-\alpha} \pi(\dd x), \quad v \leq 1 + \alpha,
\end{split}
\end{equation}
with $\alpha$ and $\beta$ defined in the paragraph below \eqref{eq:BG}.

We decompose $X_+^*$ and $X_{-}^*$ further. For $\kappa \in (0,1]$, put
\begin{equation*}
\begin{split}
A(t; \kappa) & = 
\left\{ (x,s,z): \, \frac{z}{x} \left( 1 - e^{-x(t-s)} \right) 
\leq t^{\kappa}, s \in (0,t], x > 0, z \in (0,1] \right\}, \\
B(t; \kappa) & = 
\left\{ (x,s,z): \, \frac{z}{x} e^{xs} \left( 1 - e^{-xt} \right) 
\leq t^{\kappa}, s \leq 0, x > 0, z \in (0,1] \right\},
\end{split}
\end{equation*}
and define the truncated processes
\begin{equation*}
\begin{split}
& X_{+,<,\kappa }^*(t) = \trint_{A(t;\kappa)}
\frac{z}{x} \left( 1 - e^{-x(t-s)} \right) 
(\mu - \nu)(\dd x, \dd s, \dd z), \\
& X_{+,>,\kappa}^*(t) = \trint_{A(t;\kappa)^c}
\frac{z}{x} \left( 1 - e^{-x(t-s)} \right) 
(\mu - \nu)(\dd x, \dd s, \dd z),
\end{split}
\end{equation*}
and 
\begin{equation*}
\begin{split}
& X_{-,<,\kappa}^*(t) = \trint_{B(t;\kappa)}
\frac{z}{x} e^{xs} ( 1 - e^{-xt}) 
(\mu - \nu) (\dd x, \dd s, \dd z), \\
& X_{-,>,\kappa}^*(t) = \trint_{B(t;\kappa)^c}
\frac{z}{x} e^{xs} ( 1 - e^{-xt}) 
(\mu - \nu) (\dd x, \dd s, \dd z).
\end{split}
\end{equation*}
Here, to ease notation the complement is understood as
$A(t;\kappa)^c = (0,\infty) \times (0,t] \times (0,1] \backslash A(t;\kappa)$,
and 
$B(t;\kappa)^c = (0,\infty) \times (-\infty,0] \times (0,1] \backslash B(t;\kappa)$.
Note that the truncated processes do depend on $\kappa$.
The truncation level $\kappa = 1$ means no truncation, that is $X_{\pm, <, 1}^*(t) = X_\pm^*(t)$.

\begin{lemma} \label{lemma:X<momgenfunc}
For $|\theta| \leq 4t^{-\kappa}$
\[
\log \E e^{\theta X_{+, <, \kappa}^*(t)} \leq
\begin{cases}
c \theta^2 t^{ \beta- \alpha + \kappa (2-\beta)}, & \text{if } \ 
\alpha \leq \beta -1, \\
c \theta^2 t^{1 + \kappa (1-\alpha)_+ }, & \text{if } \   
\alpha \geq \beta - 1,
\end{cases}
\]
and
\[
\log \E e^{\theta X_{-, <, \kappa}^*(t)} \leq
\begin{cases}
c \theta^2 t^{ \beta- \alpha + \kappa (2-\beta)}, 
& \text{if } \ \alpha \leq \beta, \\
c \theta^2 t^{\kappa ( 2- \alpha)_+}, & \text{if } \ \alpha \geq \beta.
\end{cases}
\]
\end{lemma}

\begin{proof}
For the moment generating function we have
\[
\E e^{\theta X_{\pm, <, \kappa}^*(t)} = \exp \left\{ 
\int \left( e^{\theta y} - 1 - \theta y \right) \Pi_t^{\pm}(\dd y)
\right\},
\]
where the L\'evy measures $\Pi_t^{\pm}$ are given by
\[
\begin{split}
\Pi_t^+(D) & = \nu \left( 
\left\{ (x,s,z) \in A(t;\kappa) : \, 
\frac{z}{x} \left( 1 - e^{-x(t-s)} \right) \in D \right\}
\right), \\
\Pi_t^-(D) & = \nu \left( 
\left\{ (x,s,z) \in B(t;\kappa) : \, 
\frac{z}{x} e^{xs} ( 1 - e^{-xt} ) \in D \right\}
\right),
\end{split}
\]
for any Borel set $D$.
The moment generating function is finite 
for any $\theta \in \R$, see e.g.~\cite[Section 26]{Sato}.
Using the inequality $e^{u} - 1 - u \leq 4u^2$ for $|u| \leq 4$, 
and that $\Pi_t^+((t^\kappa, \infty)) = 0$, 
we obtain for $|t^{\kappa} \theta | \leq 4$
\[
\begin{split}
& \log  \E e^{\theta X_{+, <, \kappa}^*(t) }  \leq 
\int 4\theta^2 y^2 \Pi_t^+(\dd y)\\
&= \trint_{(0,\infty) \times (0,t] \times (0,1]} 4\theta^2 \frac{z^2}{x^2} 
( 1 - e^{-xu})^2 
\bone \left( \frac{z}{x} \left( 1 - e^{-xu} \right) 
\leq t^{\kappa} \right) \nu(\dd x, \dd u, \dd z) \\
& \leq  4\theta^2 
\trint_{(0,\infty) \times (0,t] \times (0,1]} 
\bigg[
\bone( xu \leq 1, zu \leq 2 t^\kappa ) z^2 u^2 \\
& \hspace{4.5cm}
+ \bone \left( xu > 1, z \leq 2 t^\kappa x \right)
z^2 x^{-2} \bigg] \nu(\dd x, \dd u, \dd z) \\
& =: 4\theta^2 (I_1 + I_2).
\end{split}
\]

For $I_1$ we have from \eqref{eq:int-pi-lambda-bounds}
\[
\begin{split}
I_1 & \leq \int_{0}^1 u \dd u \int_{(0,1]} z^2 \lambda(\dd z) m_{-1}(\pi) + 
\int_{1}^t u^2 \dd u \int_{(0,u^{-1}]} \pi(\dd x) 
\int_{(0, 2t^\kappa/u]} z^2 \lambda(\dd z) \\
& \leq c \left( 1 + \int_1^{2t^\kappa} u^{1-\alpha} \dd u + 
\bone(2t^\kappa < t)
\int_{2t^\kappa}^t u^{1-\alpha + \beta -2} t^{\kappa(2-\beta)} \dd u \right) \\
& \leq 
c \, \Big( 1 + \bone( \alpha < 2) t^{\kappa ( 2-\alpha)}  + 
\bone ( \alpha  =2) \log t \\ 
& \qquad + [ \bone(\kappa < 1, \beta = \alpha) \log t  + t^{(\beta - \alpha)_+} ]
t^{\kappa (2 - \beta)} \Big),
\end{split}
\]
and for $I_2$
\[
\begin{split}
I_2 & \leq \int_{(t^{-1}, \infty)} t x^{-2}  \pi (\dd x) 
\int_{(0, 2t^\kappa x]}  z^2 \lambda(\dd z) \\
& \leq  c \bone( 2t^\kappa < t) t^{1 + \kappa(2-\beta)} \int_{(t^{-1}, t^{-\kappa}/2]} 
x^{-\beta} \pi(\dd x) + 
c t \int_{(t^{-\kappa}/2, \infty)} x^{-2} \pi(\dd x)  \\
& \leq 
c \left( \bone ( 1 + \alpha \geq \beta) t^{ 1 + \kappa(1-\alpha) } +
\bone ( 1 + \alpha \leq \beta) t^{ \beta - \alpha + \kappa(2-\beta) } 
+ t^{ 1 + \kappa(1-\alpha)_+ } \right) \\
& \leq 
\begin{cases}
c t^{\beta - \alpha + \kappa(2-\beta)}, & \alpha \leq \beta -1,\\
c t^{1 + \kappa (1-\alpha)_+}, & \alpha \geq \beta -1 , 
\end{cases}
\end{split}
\]
proving the first statement.

Similarly, for $X_{-,<,\kappa}^*$ if $|t^{\kappa} \theta | \leq 4$
\[
\begin{split}
& \log  \E e^{\theta X_{-, <, \kappa}^*(t) } \leq 
\int 4\theta^2 y^2 \Pi_t^-(\dd y) \\
& = \trint 4\theta^2 \frac{z^2}{x^2} 
e^{2xs} ( 1 - e^{-xt})^2 
\bone \left( B(t; \kappa) \right) \nu(\dd x, \dd s, \dd z) \\
& \leq  4\theta^2 
\trint_{(0,\infty) \times (-\infty,0] \times (0,1]} 
\bigg[
\bone( xt \leq 1, e^{xs} \leq 2 t^{\kappa-1} z^{-1} ) z^2 t^2 
e^{2xs} \\
& \hspace{3.5cm} 
+ \bone \left( xt > 1, e^{xs} \frac{z}{x} \leq 2 t^\kappa \right)
z^2 x^{-2} e^{2xs} \bigg] \nu(\dd x, \dd s, \dd z) \\
& =: 4\theta^2 (J_1 + J_2).
\end{split}
\]
Integrating with respect to $s$, again by \eqref{eq:int-pi-lambda-bounds} we get
\[
\begin{split}
J_1 & =  t^2 \int_{(0, t^{-1}]} x^{-1} \pi (\dd x) 
\left[ \int_{(0, 2t^{\kappa-1}]} \frac{1}{2} z^2 \lambda(\dd z) + 
\int_{(2 t^{\kappa-1}, 1]} 2 t^{2(\kappa-1)} \lambda(\dd z) \right] \\
& \leq  c t^{2 - \alpha} \left[ t^{ -  ( 2 - \beta) (1-\kappa)} + 
t^{ - (2- \beta) ( 1 -\kappa) } \right]
= c t^{\beta - \alpha + \kappa ( 2 - \beta)},
\end{split}
\]
and
\[
\begin{split}
J_2 & = \int_{(t^{-1}, \infty)} \frac{1}{2x^{3}} \pi(\dd x)
\int_{(0, 1] } z^2 \lambda(\dd z) 
\left[ \bone ( 2 t^\kappa x  \geq z ) + 
\frac{(2 t^\kappa x)^2}{z^{2}} \bone ( 2 t^\kappa x  < z ) \right] \\
& \leq c \bone(2t^\kappa < t) \int_{(t^{-1}, t^{-\kappa}/2]} x^{-3} 
(t^\kappa x)^{2 - \beta} \pi ( \dd x) +
c \int_{(t^{-\kappa}/2, \infty)} x^{-3} \pi ( \dd x) \\
& \leq c \left( \bone ( \alpha \leq \beta) t^{\beta - \alpha + \kappa(2-\beta)} + 
\bone (\alpha \geq \beta) t^{\kappa(2-\alpha)_+} 
+ t^{\kappa (2-\alpha)_+}  \right).\qedhere
\end{split}
\]
\end{proof}

\begin{corollary} \label{cor:gamma>1}
For any $\kappa \in [1/2,1]$ such that
\begin{equation*}
\kappa \geq \begin{cases}
1 - \tfrac{\alpha}{\beta}, & \text{if } \ \alpha \leq \beta -1, \\
\tfrac{1}{1 + \alpha}, & \text{if } \  \beta - 1  \leq \alpha \leq 1, \\
\tfrac{1}{2}, & \text{if } \   \alpha \geq 1, \\
\end{cases}
\end{equation*}
it holds that
\[
\limsup_{t \to \infty} \frac{|X^*_{\pm, <, \kappa}(t)|}{t^\kappa \log t}  \leq \frac{1}{4} \quad \text{a.s.}
\]
\end{corollary}

\begin{proof}
Fix $\varepsilon > 0$.
Applying Markov's inequality together with Lemma \ref{lemma:X<momgenfunc}
with $\theta = 4t^{-\kappa}$, we obtain
\begin{equation} \label{eq:Markov-X<-bound}
\begin{split}
& \p ( X_{+, <, \kappa} ^*(t) > (1/4 + \varepsilon) t^\kappa \log t  )  
= \p \left( e^{\theta X_{+, <, \kappa}^*(t) } > 
e^{\theta (1/4 + \varepsilon) t^\kappa \log t } \right) \\
& \leq 
\begin{cases}
\exp \left\{ c \theta^2 t^{\beta - \alpha + \kappa(2-\beta)} - \theta 
(1/4 + \varepsilon) t^\kappa \log t \right\} 
\leq c t^{-(1+4\varepsilon)},
& \alpha \leq \beta -1, \\
\exp \left\{ c \theta^2 t^{1 + \kappa (1-\alpha)_+} - \theta 
(1/4 + \varepsilon) t^\kappa \log t \right\} \leq c t^{-(1+4\varepsilon)},
& \alpha \geq \beta - 1,  
\end{cases}
\end{split}
\end{equation}
where we used that $\beta - \alpha + \kappa(2 -\beta) - 2 \kappa \leq 0$
for $\alpha \leq \beta -1$, and $1 + \kappa ( 1 - \alpha )_+ - 2 \kappa \leq 0$
for $\alpha \geq \beta - 1$.
The same bound for $- X_{+, <, \kappa}^*(t)$ combined with 
the Borel--Cantelli lemma shows that the convergence takes place on the 
subsequence $t_n = n$. From the proof of Lemma \ref{lemma:X+diffbound} it is clear 
that the statement remains true for $X_{\pm, <, \kappa}^*(t)$,
thus the result follows by Lemma \ref{lemma:aux-conv}. 

For $X_{-, <, \kappa}^*$ the same argument applies. In the corresponding version 
of \eqref{eq:Markov-X<-bound} one has to check that 
$\beta - \alpha + \kappa(2 - \beta) - 2 \kappa \leq 0$ for $\alpha \leq \beta$,
and $\kappa ( 2 - \alpha)_{+} - 2 \kappa \leq 0$ for $\alpha \geq \beta$.
Short calculations give that these inequalities hold whenever $\kappa$ satisfies 
the assumptions of the statement.
\end{proof}

\begin{lemma} \label{lemma:X>-0}
If for some $\gamma \geq 1$
\begin{equation*}
\dint_{(0, \infty) \times (0,1]}  \left( \frac{z}{x} \right)^\gamma 
\bone(z > x) \pi(\dd x) \lambda(\dd z) < \infty,
\end{equation*}
then almost surely $\mu(A(t;1/\gamma)^c)=0$ and $\mu(B(t;1/\gamma)^c)=0$ for $t$ large enough.
\end{lemma}

\begin{proof}
We have
\[
\nu( \{ (x,s,z): z > x n^{1/\gamma}, s \in (n, n+1) \} ) = 
\dint \bone ( z > x n^{1/\gamma} )  \pi(\dd x) \lambda(\dd z),
\]
thus
\[
\begin{split}
& \sum_{n=1}^\infty \nu( \{ (x,s,z): z > x n^{1/\gamma}, s \in (n, n+1) \} ) \\
& = \dint \sum_{n=1}^\infty \bone ( z > x n^{1/\gamma} ) 
\pi(\dd x) \lambda(\dd z) \\
& \leq \dint \left( \frac{z}{x} \right)^\gamma \bone(z > x) 
\pi(\dd x) \lambda(\dd z) < \infty,
\end{split}
\]
by the assumption. 
Therefore, the Borel--Cantelli lemma implies that for $n$ large enough
\begin{equation*}
\mu( \{ (x,s,z): z > x n^{1/\gamma}, |s| \in (n, n+1) \} ) = 0,
\end{equation*}
from which the statement follows.
\end{proof}

Next we handle the compensation of $X_{\pm, >, \kappa}^*$.

\begin{lemma} \label{lemma:compX>}
For any $\kappa \in (0,1)$
\begin{equation} \label{eq:comp-bound+}
\trint_{A(t;\kappa)^c}
\frac{z}{x} \left( 1 - e^{-x(t-s)} \right) 
 \nu(\dd x, \dd s, \dd z)  \leq 
\begin{cases}
c t^{\beta - \alpha - \kappa (\beta-1)}, & \alpha \leq \beta -1, \\
c t^{1 - \kappa \alpha}, &  \alpha \geq \beta -1,
\end{cases}
\end{equation}
and 
\begin{equation} \label{eq:nu-Acbound}
\nu ( A(t; \kappa)^c ) \leq 
\begin{cases}
c t^{\beta - \alpha - \kappa \beta}, & \alpha \leq \beta -1, \\
c t^{1 - \kappa (1+\alpha)}, &   \alpha \geq \beta -1, \\
\end{cases}
\end{equation}
while for the negative jumps
\begin{equation} \label{eq:comp-bound-}
\trint_{B(t;\kappa)^c}
\frac{z}{x} ( 1 - e^{-xt} ) e^{xs} 
\nu(\dd x, \dd s, \dd z)  \leq c t^{\beta - \alpha - \kappa (\beta-1)},
\end{equation}
and
\begin{equation} \label{eq:nu-Bcbound}
\nu ( B(t; \kappa)^c ) \leq 
c \log t \,  t^{\beta - \alpha - \kappa \beta}.
\end{equation}
\end{lemma}

\begin{proof}
Let $K$ denote the left-hand side in \eqref{eq:comp-bound+}. 
Then
\[
\begin{split}
K & \leq \trint
\left[ z u \bone( ux \leq 1, zu > t^\kappa) 
+ \frac{z}{x} \bone( ux > 1, z > t^\kappa x) \right] \nu(\dd x, \dd u, \dd z) \\
& =: I_1 + I_2.
\end{split}
\]
Since in $I_1$ necessarily $u \geq t^\kappa$, we have by \eqref{eq:int-pi-lambda-bounds}
\[
\begin{split}
I_1  & \leq \int_{t^\kappa}^t u \dd u \int_{(0,u^{-1}]} \pi(\dd x) \int_{(t^\kappa/u, 1]} z \lambda (\dd z) \\ 
& \leq c t^{-\kappa(\beta-1)} \int_{t^\kappa}^t u^{\beta-1-\alpha} \dd u \leq 
\begin{cases} 
c t^{\beta - \alpha - \kappa (\beta -1)},& \hspace{-2pt}\alpha < \beta, \\
c \log t \, t^{\kappa ( 1-\alpha)},& \hspace{-2pt}\alpha = \beta, \\
c t^{\kappa ( 1 -  \alpha)},& \hspace{-2pt}\alpha < \beta.
\end{cases}
\end{split}
\]
Similarly, for $I_2$ after integration with respect to $s$ we have
\[
I_2  \leq t \int_{(t^{-1}, t^{-\kappa}]} x^{-1} \pi(\dd x)
\int_{(t^\kappa x, 1]} z \lambda(\dd z)  \leq 
\begin{cases}
c t^{\beta - \alpha - \kappa(\beta - 1)}, & \alpha \leq \beta -1, \\
c t^{1- \kappa \alpha}, & \alpha \geq \beta -1, 
\end{cases}
\]
and \eqref{eq:comp-bound+} follows. 
In exactly the same way we obtain \eqref{eq:nu-Acbound}. The only difference 
is that in both $I_1$ and in $I_2$ a factor $t^{-\kappa}$ appear.

\smallskip

Denoting by $L$ the left-hand side of \eqref{eq:comp-bound-}, we similarly have
\[
\begin{split}
L & \leq \trint 
\bigg[ 
z t e^{xs} \bone (xt \leq 1, zt e^{xs} > t^\kappa) \\
& \hspace{1.5cm} + 
\frac{z}{x} e^{xs} \bone (xt> 1,  z e^{xs} > t^\kappa x) \bigg] 
\nu(\dd x, \dd s, \dd z)  \\
& = : J_1 + J_2.
\end{split}
\]
In both $J_1$ and $J_2$ we bound the $\dd s$ integral by $x^{-1}$,
thus we obtain by using \eqref{eq:int-pi-lambda-bounds}
\[
\begin{split}
J_1 & \leq t \int_{(0, t^{-1}]} x^{-1} \pi(\dd x) 
\int_{(t^{\kappa-1}, 1]} z \lambda(\dd z) 
\leq c t^{\beta - \alpha - \kappa(\beta - 1)},
\end{split}
\]
and 
\[
\begin{split}
J_2 & \leq \int_{(t^{-1}, t^{-\kappa}]} x^{-2} \pi(\dd x) 
\int_{(t^\kappa x, 1]} z \lambda(\dd z) 
\leq 
\begin{cases}
c t^{\beta - \alpha - \kappa(\beta -1)}, & \alpha \leq \beta, \\
c t^{\kappa ( 1- \alpha)}, & \alpha \geq \beta,
\end{cases}
\end{split}
\]
proving \eqref{eq:comp-bound-}. The estimate 
\eqref{eq:nu-Bcbound} follows in the same way, except that we bound 
the $\dd s$ integral in both terms by $(1- \kappa) x^{-1} \log t$.
\end{proof}

\begin{lemma} \label{lemma:X>betanagy}
Assume that $1 + \alpha \leq \beta$. Then for 
any $\kappa > 1 - \tfrac{\alpha}{\beta}$ almost surely
$\mu(A(t; \kappa)^c) = 0$ and 
$\mu(B(t; \kappa)^c) = 0$ for $t$ large enough.
\end{lemma}

\begin{proof}
Fix $\kappa' \in (1 - \alpha/\beta, \kappa)$.  Then for some 
$\delta = \delta(\kappa') > 0$, 
by \eqref{eq:nu-Acbound} in Lemma \ref{lemma:compX>}
\[
\nu ( A(t; \kappa')^c ) \leq c t^{-\delta}.
\]
Thus if $d \delta > 1$, by the Borel--Cantelli lemma
$\mu ( A(n^d, \kappa')^c) = 0$ a.s.~for $n$ large enough. Assume 
that $\mu ( A(t; \kappa)^c ) \geq 1$ for some $t \in ((n-1)^d, n^d]$.
Then there exists a point $(\xi, \tau, \zeta) \in A(t; \kappa)^c$, such that, if $n$ is large enough,
\[
t^\kappa \leq \frac{\zeta}{\xi} \left( 1 - e^{-\xi(t-\tau)} \right) \leq 
\frac{\zeta}{\xi} \left( 1 - e^{-\xi(n^d-\tau)} \right) \leq 
n^{d \kappa'},
\]
which is a contradiction. Therefore 
$\mu(A(t;\kappa)^c) = 0$ for $t$ large enough, as claimed.

The same proof works for $B(t;\kappa)^c$ since the extra $\log t$ 
factor in \eqref{eq:nu-Bcbound} plays no role here.
\end{proof}

We now combine the previous results to obtain the statement for 
$\lambda$ supported on $[-1,1]$.

\begin{corollary}\label{cor:zleq1}
Assume that $\int_{|z|\leq 1} |z| \lambda(\dd z)=\infty$, $a=0$, $b=0$
in \eqref{eq:Lambdace} and $\lambda$ is 
supported on $[-1,1]$. Then 
\begin{equation} \label{eq:lsgen}
\limsup_{t \to \infty} \frac{|X^*(t)|}{t \log t} \leq 1 \quad \text{a.s.}
\end{equation}
Furthermore, if $\beta \leq 1 + \alpha$ and for some $\gamma \in [1,2]$
\begin{equation} \label{eq:gammacond}
\int_{(0,\infty)} \int_{|z|\leq 1} \frac{|z|^\gamma}{x^\gamma} \bone(|z| > x) \lambda(\dd z) \pi (\dd x) < \infty,
\end{equation}
then 
\[
\limsup_{t \to \infty} \frac{|X^*(t)|}{t^{1/\gamma} \log t} \leq 1 \quad \text{a.s.}
\]

If $\beta \geq 1 + \alpha$, then for any $\gamma < 1/(1 - \tfrac{\alpha}{\beta})$ 
\[
\lim_{t \to \infty} \frac{X^*(t)}{t^{1/\gamma}} = 0  \quad \text{a.s.}
\]
\end{corollary}

\begin{proof}
Assume first that $\lambda$ is supported on $(0,1]$. Then \eqref{eq:lsgen} follows from Corollary \ref{cor:gamma>1} with $\kappa = 1$. Recall that $\kappa = 1$ corresponds to no truncation, thus $X_{\pm,<,1}^*(t)=X_{\pm}^*(t)$.

Next assume \eqref{eq:gammacond}. Then $1+\alpha \geq \gamma$, and 
by assumption $\beta \leq 1+\alpha$. Lemma \ref{lemma:X>-0} implies that the terms
$X_{\pm,>,1/\gamma}^*(t)$ do not contribute to the limit. The result now follows from Corollary \ref{cor:gamma>1} and Lemma \ref{lemma:compX>} by taking $\kappa = 1/\gamma$.

For $\beta \geq 1 + \alpha$ the result follows from Corollary \ref{cor:gamma>1} and Lemma \ref{lemma:compX>} with $\kappa = 1 - \tfrac{\alpha}{\beta}$, and from Lemma \ref{lemma:X>betanagy}. 

For $\lambda$ supported on $[-1,1]$, we apply the previous to $\mu\bone(z>0)$ and $\mu\bone(z<0)$ separately.
\end{proof}

\begin{proof}[Proof of Theorem \ref{thm:infv}]
If we make a decomposition based on $\mu\bone(|z|\leq 1)$ and $\mu\bone(|z|>1)$, we obtain the results by applying Corollary \ref{cor:zleq1} and Theorem \ref{thm:fv}, respectively.
\end{proof}

\section{Proof of Theorem 4}\label{sec:proof-Gaussian}

We will use a general result for Gaussian processes from \cite{orey1972}. Since $X$ is 
Gaussian, we have that 
\[
\E e^{\ii \theta \Lambda(A)} = e^{-\frac{b}{2} \theta^2 (\pi \times \Leb (A))}
\]
and by 
Proposition 2.6.~in \cite{rajputrosinski1989} 
\[
\log \E e^{\ii \theta X^*(t)} 
= - \frac{b}{2} \theta^2 \int_{(0,\infty)} \int_0^{t}
\left( 1 -e^{-x u} \right) x^{-2} \dd u \pi(\dd x).
\]
Thus $X^*$ is a mean 0 Gaussian process with 
\[
Q(t) := \frac{1}{2} \Var(X^*(t)) = 
\frac{b}{2} \int_{(0,\infty)} \int_0^{t}
\left( 1 -e^{-x u} \right) x^{-2} \dd u 
\pi(\dd x).
\]
Hence, $Q$ is non-decreasing and $Q(0)=0$. 

\begin{enumerate}
\item If $m_{-2}(\pi)<\infty$, then
\begin{equation*}
Q(t) = \frac{b}{2} t \int_{(0,\infty)}\frac{1+tx-e^{-tx}}{tx} x^{-2} \pi(\dd x) \sim \frac{b}{2} t m_{-2}(\pi),
\end{equation*}
as $t\to \infty$. Now the result follows by Theorem 1.1 in \cite{orey1972} with $v(t)=\frac{b}{2} t m_{-2}(\pi)$ (see also Remark 1.3 in \cite{orey1972}).

\item As in \cite[Eq.~(5.8)]{GLT2019Limit}, we can 
write $Q(t)$ in the form
\[
Q(t) =  \frac{b}{2} \frac{\Gamma(1+\alpha)}{(2-\alpha)(1-\alpha)} \ell_1(t)  t^{2-\alpha},
\]
with $\ell_1$ slowly varying at infinity such that $\ell_1(t) \sim \ell(t)$. By Potter's 
bounds \cite[Theorem 1.5.6]{BGT}, for $\delta < \alpha/3$, there is $s_0>0$ such that
\[
\left(\frac{t}{s}\right)^{2-\alpha-\delta} Q(s) \leq Q(t) \leq 
\left(\frac{t}{s}\right)^{2-\alpha+\delta} Q(s), \quad t\geq s >s_0.
\]
Moreover,
\[
Q(t) \leq \frac{b}{2} \int_{(0,\infty)} \int_0^{t} u \dd u x^{-1} \pi(\dd x) = \frac{b}{4} 
m_{-1}(\pi) t^2.
\]
Thus, the assumptions of Theorem 1.1 in \cite{orey1972} hold with $v=Q$.
\end{enumerate}

\section*{Appendix A: A Fubini theorem}\label{app:A}

Here we show that the stochastic Fubini theorem from \cite{bnbasse2011} 
is indeed applicable in our setup.
If $\pi$ is a finite measure, this was justified in \cite[Lemma 4.1]{GLT21} (see also \cite[Lemma 3.1]{talarczyk2020} for a similar argument).

Put $g_t(x,s) = e^{-x(t-s)} \bone (s \leq t)$.
If $\E |X(1)| < \infty$, then we can directly use a stochastic Fubini theorem given in \cite[Theorem 3.1]{bnbasse2011} which guarantees the integrals exist. Indeed, the conditions of Theorem 3.1 and Remark 3.2 in \cite{bnbasse2011} boil down to showing that \\
(i) for every $u\in [0,t]$, $g_u(\cdot, \cdot)$ is in the Musielak-Orlicz space $L^{\phi_1}$, that is
\begin{equation*}
\dint_{(0,\infty)\times \R} \left(b g_u(x, s)^2 + 
\int_{\R} (|z g_u(x,s)|^2 \wedge |z g_u(x, s)|) \lambda(\dd z) \right) \pi(\dd x) \dd s < \infty,
\end{equation*}
(ii) it holds that
\begin{equation*}
\begin{split}
\int_0^t \bigg[ \dint_{(0,\infty)\times \R} \Big( & b g_u(x, s)^2 \\
& + \int_{\R} (|z g_u (x ,s)|^2 \wedge |z g_u (\xi, s)|) \lambda (\dd z) 
\Big) \pi(\dd x) \dd s \bigg] \dd u < \infty.
\end{split}
\end{equation*}

By \cite[Theorem 3.3]{rajputrosinski1989}, $L^{\phi_1}$ coincides with the space of 
$\Lambda$-integrable functions $g$ such that $\E |\int g d\Lambda|<\infty$. The supOU process is well defined, i.e.~$g_t(\cdot,\cdot)$ is $\Lambda$-integrable and since we have assumed $\E |X(u)|<\infty$, we conclude that condition (i) holds. By the change of variables, first $r=s-u$ and then $r=t-s$
\[
\begin{split}
&\int_0^t \int_{(0,\infty)} \int_{\R} \bigg(b e^{-2x (u - s)} \bone (s\leq u) \\
& \hspace{2.5cm} +\int_{\R} (|z e^{-x (u - s)}|^2 \wedge |z e^{-x (u - s)} |) \bone (s\leq u) \lambda(\dd z) \bigg) \pi(\dd x) \dd s \dd u\\
&=\int_0^t \int_{(0,\infty)} \int_{\R}  \bigg(b e^{2x r} \bone (r\leq 0) \\
& \hspace{2.8cm} + \int_{\R} (|z e^{x r}|^2 \wedge |x e^{x r} |) \bone (r\leq 0) \lambda (\dd z) \bigg) \pi(\dd x) \dd r \dd u\\
&=t \int_{(0,\infty)} \int_{\R} \bigg(b e^{2x r} \bone (r\leq 0) + \int_{\R} (|z e^{x r}|^2 \wedge |z e^{x r} |) \bone (r\leq 0) \lambda(\dd z) \bigg) \pi(\dd x) \dd r\\
&=t \int_{(0,\infty)} \int_{\R}\bigg(b e^{-2x (t-s)} \bone (s\leq t) \\
& \hspace{2.5cm} + \int_{\R} (|x e^{-x (t-s)}|^2 \wedge |x e^{-x (t-s)} |) \bone (s\leq t) \lambda(\dd z) \bigg) \pi(\dd x) \dd r\\
&=\int_{(0,\infty)} \int_{\R} \left(b g_t(x, s)^2 + \int_{\R} (|x g_t(x,s)|^2 \wedge |x g_t(x, s)|) \lambda(\dd z) \right) \pi(\dd x) \dd s < \infty,
\end{split}
\]
hence, (ii) follows from (i).

Suppose now that $\E |X(1)| = \infty$. By the L\'evy-It\^o decomposition of $\Lambda$, we may write it in the form
\begin{equation*}
\Lambda (A) = \Lambda_1 (A) + \Lambda_2(A),
\end{equation*}
where $\Lambda_1$ and $\Lambda_2$ are independent random measures such that $\Lambda_1$ has the characteristic quadruple $(0,0,\lambda_1,\pi)$, $\lambda_1(\dd z)=\lambda(\dd z) \bone (|z|>1)$ and $\Lambda_2$ has characteristic quadruple $(a,b,\lambda_2,\pi)$, $\lambda_2(\dd z)=\lambda(\dd z) \bone (|z|\leq 1)$. By \cite[Proposition 3.3.9]{samorodnitsky2016} we have that
\begin{equation*}
\begin{split}
& \int_{(0,\infty) \times \R} g_u(x, s) \Lambda(\dd x, \dd s) \\
& = \int_{(0,\infty)  \times \R} g_u(x, s) \Lambda_1(\dd x, \dd s) + \int_{(0,\infty)  \times \R} g_u(x, s) \Lambda_2(\dd x, \dd s), \ \text{a.s.}
\end{split}
\end{equation*}
For the integral with respect to $\Lambda_2$ we can apply \cite[Theorem 3.1]{bnbasse2011} as in the previous case. It remains to consider $\Lambda_1$, which is a compound Poisson random measure. If $(\xi_k,\tau_k,\zeta_k)_{k\geq 0}$ are the points of $\Lambda_1$, then the corresponding supOU process $X_1(u)$ can be written as
\begin{equation*}
X_1(u) = \sum_{\tau_k\leq 0} \zeta_k e^{-\xi_k (u- \tau_k)} + \sum_{0<\tau_k\leq u} \zeta_k e^{-\xi_k (u- \tau_k)}.
\end{equation*}
The second sum has finitely many terms a.s.~and for the first one we have
\begin{equation*}
\int_{0}^t \sum_{\tau_k\leq 0} |\zeta_k| e^{-\xi_k (u- \tau_k)} \dd u \leq t \sum_{\tau_k\leq 0} |\zeta_k| e^{\xi_k \tau_k},
\end{equation*}
which is finite a.s.~since
\begin{equation*}
\trint_{(0,\infty) \times (-\infty, 0) \times \R} (|z| e^{x s} \wedge 1) \pi(\dd x) \dd s \mu_1(\dd z)<\infty.
\end{equation*}

\section*{Appendix B: An exotic example}\label{app:B}

Here we give an example for $\pi$ and $\lambda$ such that $\beta > 1 + \alpha$ and 
\begin{equation*}
\dint_{(0,\infty) \times \R} \frac{|z|^\gamma}{x^\gamma} \bone ( |z| > x) 
\pi(\dd x) \lambda(\dd z)   < \infty,
\end{equation*}
still holds with $\gamma > 0$.
Let $a_n = 2^{-2^n}$ and $a \geq 0$. 
For $r \in (a_n, a_{n-1}]$ put $\pi( (0,r] ) = a_n^{1+a}$.
Then
$\pi (\{ a_n  \} ) = a_n^{1+a} - a_{n+1}^{1+a} \sim a_n^{1+a}$, thus
\[
\int_{(0,1]} x^{-1 - \alpha} \pi (\dd x) = \sum_{n=1}^\infty a_n^{-1-\alpha} \pi(\{a_n\})
< \infty
\]
if and only if $a > \alpha$. In particular, $\alpha_0 = a$ and $\alpha < a$.

Let $b > 0$ and  for $r \in (a_n ,a_{n-1}]$ let 
$\overline \lambda(r) = \lambda((r,1]) = a_{n-1}^{-b}$.
Then $\lambda( \{ a_n \}) = a_n^{-b} - a_{n-1}^{-b} \sim a_n^{-b}$,
therefore,
\[
 \int_{(0,1]} z^\beta \lambda(\dd z) = \sum_{n=1}^\infty a_n^{\beta}  \pi(\{a_n\}) < \infty
\]
if and only if $b < \beta$. Thus $\beta_0 = b$ and $\beta > b$.
Furthermore, for $\gamma < b$, $r \in (0,1)$
\[
\int_{(r,1]} z^\gamma \lambda(\dd z) \sim \sum_{a_n > r} a_n^{\gamma - b}
\sim a_{\psi(r)}^{\gamma - b},
\]
where $\psi(r) = \max \{ k \colon a_k > r \}$.

Thus 
\[
\dint \frac{z^\gamma}{x^\gamma} \bone ( z > x) \lambda (\dd z) \pi (\dd x)
\sim \sum_n a_n^{1+a - \gamma} a_{\psi(a_n)}^{\gamma - b} =
\sum_n a_{n-1}^{2(1+a - \gamma) + \gamma - b}, 
\]
where we used that $\psi(a_n) = n-1$ and $a_n = a_{n-1}^2$. The latter sum is 
finite if and only if $2(1+a - \gamma) + \gamma - b > 0$, which is equivalent to
$\gamma < 2(1+a) - b$. This is possible even if $b > 1 + a$.

\bigskip

\noindent
\textbf{Acknowledgement.} 
We thank the referees and the AE for constructive remarks and suggestions which 
greatly improved our paper.
DG was partially supported by the Croatian Science Foundation 
(HRZZ) grant Scaling in Stochastic Models (IP-2022-10-8081).
PK was supported by the J\'{a}nos Bolyai Research Scholarship of the Hungarian Academy of Sciences, and by the NKFIH grant FK124141.


\end{document}